\documentclass{ws-sd}

\usepackage{graphicx}
\usepackage{amsmath,amssymb,amsthm,amsfonts}
\usepackage{mathrsfs}

\usepackage{amssymb}
\usepackage{varioref} 

\newtheorem{thm}{Theorem}[section]

\newtheorem{lem}[thm]{Lemma}

\theoremstyle{definition}
\newtheorem{defn}{Definition}[section]

\newcommand{\pprec}{\prec\mkern-10mu\prec}

\begin{document}

\markboth{Qi Zhang}{Global well-posedness for the nonlinear PAM equation}

\catchline{}{}{}{}{}

\title{Global well-posedness for the nonlinear generalized parabolic Anderson model equation }

\author{Qi Zhang}

\address{Yanqi Lake Beijing Institute of Mathematical Sciences and Applications,\\ Beijing, 101408, China\\
Yau Mathematical Sciences Center, Tsinghua University,\\ Beijing 100084, China\\
qzhang@bimsa.cn}

\maketitle

\begin{history}
\received{(Day Month Year)}
\revised{(Day Month Year)}
\end{history}

\begin{abstract}
We study the global existence of the singular nonlinear parabolic Anderson model equation on $2$-dimensional tours $\mathbb{T}^2$. The method is based on paracontrolled distribution and renormalization. After split the original nonlinear parabolic Anderson model equation into two simple equations, we prove the global well-posedness by some a priori estimates and smooth approximations. Furthermore, we prove the uniqueness of the solution by using classical energy estimates.
\end{abstract}



\section{Introduction}

We study the following $2$-dimensional nonlinear parabolic Anderson model (PAM) equation
\begin{equation}\label{GPAM}
    \partial_t u + (- \Delta+\mu) u = f(u) + u \diamond \xi, \quad u(0)=u_0.
\end{equation}
where $\mu>0$, $u:\mathbb{R}^{+}\times \mathbb{T}^2 \rightarrow \mathbb{R}$, the nonlinear function $f(s) = \sum_{1\leq i \leq k-1} a_i s^i$ is  a polynomial function from $\mathbb{R}$ to $\mathbb{R}$, and $\xi$ is a spatial white noise on the $2$-dimensional torus $\mathbb{T}^2=(\mathbb{R}/\mathbb{Z})^2$.

The Anderson model was originally introduced by Anderson [\refcite{A1958}] as a mathematical description for the electron motion in disordered medium, such as a random potential. In this famous work, Anderson showed that the electron is trapped and remain localized in a random medium. This phenomenon is called Anderson localization in condensed matter physics.

When the spatial dimension $n\geq 2$, the parabolic Anderson model equation is a kind of typical singular stochastic partial differential equation. Even though the classical stochastic partial differential equation theory has great achievements in recent decades, many stochastic partial differential equations from physics are singular and hard to deal with by the classical methods, such as the parabolic Anderson model equation, the Kardar–Parisi–Zhang (KPZ) equation, and the $\Phi^4_d$ equation. The difference between singular stochastic partial differential equations and classical stochastic partial differential equations is that the noise in singular stochastic partial differential equations is very rough. Thus the rigorous interpretation of singular stochastic partial differential equations had been an open problem for a long time.

In order to study singular stochastic partial differential equations, some new mathematical theories, such as regularity structures by Harier [\refcite{H2014}] or paracontrolled distributions by Gubinelli, Imkeller and Perkowski [\refcite{GIP2015}], had been developed in recent years. Paracontrolled distributions and regularity structures allow a pathwise description of the singular stochastic partial differential equations. In this paper, we study the 2-dimensional nonlinear parabolic Anderson model equation in the paracontrolled distribution frameworks. Comparing with regularity structures, the paracontrolled distribution approach relies on classical PDE techniques, including Littlewood-Paley decomposition, Besov space, paraproduct calculus, and develops on ideas from the theory of controlled rough paths. So it is natural and easy to use some classical PDE tools to study the parabolic Anderson model equation in paracontrolled distribution framework.

The discrete parabolic Anderson model has been well understood during the past decades, has seen in the surveys [\refcite{CM1994},\refcite{K2015}], and references therein. The well-posedness of a continuous parabolic Anderson model equation was also given in [\refcite{GIP2015}, \refcite{H2014}, \refcite{HL2015}] by different methods, including regularity structures, paracontrolled distribution, and the transformation method and a elaborate renormalisation procedure. Parabolic equations with other types of purely spatial noise potentials were studied in [\refcite{H2002}, \refcite{HW2021}, \refcite{LS2009}] by Wiener chaos decomposition. We also refer to [\refcite{GH2017} \refcite{KL2017}] for some solution properties of the parabolic Anderson model equation. In [\refcite{CGP2015}], Chouk, Gairing and Perkowski showed that the solution of a continuous parabolic Anderson model is the universal continuum limit of the 2-dimensional lattice discrete Anderson model.

The parabolic Anderson model equation can also be viewed as a heat equation with a spatial white noise potential $\xi$. Thus the parabolic Anderson model equation is also a linear parabolic equation with the Anderson Hamiltonian $\mathscr{H}$,  defined as $\mathscr{H}u:=\Delta u + u\diamond\xi$. The construction and spectrum of the Anderson Hamiltonian on $\mathbb{T}^2$ and $\mathbb{T}^3$ were studied by Allez and Chouk [\refcite{AC2015}] and Labb\'e [\refcite{L2019}]. The semilinear Schr\"{o}dinger equations and wave equations for the Anderson Hamiltonian in two and three dimensions on $\mathbb{T}^2$ and $\mathbb{T}^3$ have been considered in [\refcite{GUZ2020}]. In [\refcite{ZD2021}], we also consider the variation problem associated with the Anderson Hamiltonian  in the paracontrolled distribution framework.

Even though the local well-posedness results of paracontrolled solution for generalized parabolic Anderson model  equation were given in [\refcite{BDH2015}, \refcite{GIP2015}] by fixed point argument, there are still some difficulties to obtain the global well-posedness in paracontrolled approach. In recent years, the global well-posedness of the $\Phi^4_3$ equation was proved in [\refcite{AK17}, \refcite{GH2019}, \refcite{MW2017}]. In these works, the norm of solution was estimated by using the dissipative property of nonlinear term.

In this present chapter, we study the global well-posedness of the nonlinear parabolic Anderson model equation in paracontrolled distribution framework. We assume that the nonlinear term $f(u)$ satisfies the following dissipative assumption:
For every $s\in\mathbb{R}$,
\begin{align}\label{assumf}
    -C_0- C_1|s|^k \leq & f(s)s \leq C_0 - C_2|s|^k, \quad k\geq 3,  \nonumber  \\
     f'(s) \leq & l,
\end{align}
where $C_0, C_1, C_2, l>0$ are positive constants.

In order to define the singular term $u\diamond\xi$, we carry out the renormalization procedure and paracontrolled distribution. Then we decompose solution into two parts: $u = \phi + \psi$, and we use a localization technique which developed from [\refcite{GH2019}] to split the original singular stochastic partial differential equation in two simple equations:
\begin{equation}\label{decGPAM1}
  \left\{
   \begin{aligned}
   & \partial_t \phi + (- \Delta+\mu)\phi = \Phi, \quad \phi(0) = \phi_0=u_0,\\
   & \partial_t \psi + (- \Delta+\mu)\psi = f(\psi)+ \Psi, \quad \psi(0)=0,
   \end{aligned}
   \right.
\end{equation}
where $\Phi$ contains all of irregular but linear terms, and $\Psi$ contains all the regular terms and the nonlinear terms. By this way, we can handle the irregular part $\phi$ by paracontrolled distribution arguments, and we can analysis the regular part $\psi$ by some classical PDE methods.
Since the regularity of initial value $u_0$ is low, we also introduce a time weight $\tau(t):=1-e^{-t}$ to control the singularity when $t$ is small. Combining with the dissipative assumption (\ref{assumf}) of nonlinear term $f$, we establish the parabolic Schauder estimates and parabolic coercive estimates, and obtain some a priori estimates under some time weights. Then we prove the global existence of solution by a smooth approximation and Aubin-Lions argument. We also show the uniqueness of solution by direct energy estimates.

We now state our global well-posedness result. We refer to Section 3.2, Theorem \ref{Gexistence} for the details of global existence, and Section 3.3, Theorem \ref{Uniq} for uniqueness result.

\begin{thm}\label{WP}
Let $u_0 \in \mathscr{C}^{-1}$, and $\alpha \in (2/3,1)$. We denote $\rho =\tau^{1+1/(k-2)+(3\alpha-2)/2}$ for a time weight. Let $\vartheta = (-\Delta + \mu)^{-1}\xi$. Then there exists a solution $(\phi, \psi)$ to system (\ref{decGPAM1}) with
\begin{equation*}
    (\phi,\psi) \in [C_{\tau^{1/(k-2)+\alpha/2}}\mathscr{C}^{\alpha} \cap C^{\alpha/2}_{\tau^{1/(k-2)+\alpha/2}}L^{\infty}] \times [C_{\rho}\mathscr{C}^{3\alpha} \cap C^1_{\rho}L^{\infty} \cap C_{\tau^{1/(k-2)}}L^{\infty}],
\end{equation*}
such that $u=\phi+\psi$ is a unique global paracontrolled solution to the nonlinear parabolic Anderson model equation (\ref{GPAM}).
\end{thm}

Throughout the chapter, we use the notation $a\lesssim b$ if there exists a constant $C>0 $, independent of the variables under consideration, such that $a \leq C\cdot b$, and we write $a\simeq b$ if $a\lesssim b$ and $b \lesssim a$.  We also use the notation $C_{x}$ to emphasize that the constant $C$ depends on the quantities $x$. The Fourier transform on the torus $\mathbb{T}^d$ is defined with $\hat{u}(k):=\mathscr{F}_{\mathbb{T}^d}u(k)= \sum_{k\in\mathbb{Z}^{d}}e^{2\pi i k \cdot x}u(x)$, so that the inverse Fourier transform on the torus $\mathbb{T}^d$ is given by $\mathscr{F}^{-1}_{\mathbb{T}^d}e^{-2\pi i k \cdot x}\hat{u}(k) = \sum_{k \in \mathbb{Z}^d}\hat{u}(k)$. The space of Schwartz functions on $\mathbb{T}^d$ is denoted by $\mathcal{S}(\mathbb{T}^d)$ or $\mathcal{S}$. The space of tempered distributions on $\mathbb{T}^d$ is denoted by $\mathcal{S}'(\mathbb{T}^d)$ or $\mathcal{S}'$. We denote $\mathscr{L}:= -\Delta +\mu$, and $\rho =\tau^{1+1/(k-2)+(3\alpha-2)/2}$.

This paper is organized as follows:  In Section 2, we revisit some basic notation and estimates of the singular SPDEs. In Section 3, we obtain some a priori estimates. Then we prove the global well-posedness by a smooth approximation. Using energy estimate, we further prove the uniqueness of solution. This paper ends with some summary and discussion in Section 4.

\section{Preliminaries}
\subsection{Besov space and Bony's paraproduct}
In this subsection, we introduce some basic notations and useful estimates about Littlewood-Paley decomposition, Besov space and Bony's paraproduct. For more details, we refer to [\refcite{BCD2011},\refcite{GIP2015}, \refcite{GH2019}].

Littlewood-Paley decomposition can describe the regularity of (general) functions via the decomposition of a (general) function into a series of smooth functions with different frequencies. In order to do this, we introduce the following dyadic partition.

Let $\varphi: \mathbb{R}^d \rightarrow [0,1]$  be a smooth radial cut-off function so that
\begin{equation*}
  \varphi(x)= \left\{
   \begin{aligned}
   & 1, \quad \quad \quad |x| \leq 1 \\
   & \text{smooth} ,  1<|x|<2 \\
   & 0, \quad \quad \quad |x|\geq 2.
   \end{aligned}
   \right.
\end{equation*}
Denote $\varrho(x) = \varphi(x)-\varphi(2^{-1}x)$ and $\chi(x) = 1- \sum_{j\geq 0}\varrho(2^{-j}x) $. Then $\chi, \varrho \in C^{\infty}_c (\mathbb{R}^d)$ are nonnegative radial functions, so that
\begin{enumerate}
\item $supp(\chi) \subset B_1(0)$ and $supp(\varrho)\subset \{x \in \mathbb{R}^d: \frac{1}{2}\leq |x| \leq 2 \}$;
\item $\chi(x)+\sum_{j\geq 0}\varrho(2^{-j}x)=1, \quad x\in \mathbb{R}^n$;
\item $supp(\chi)\cap supp(\varrho(2^{-j}x)) = \emptyset $ for $j\geq 1$ and $supp(\varrho(2^{-i}x)) \cap supp(\varrho(2^{-j}x)) = \emptyset $ for $|i-j|\geq 2$.
\end{enumerate}

\begin{defn}
For $u \in \mathcal{S}'(\mathbb{T}^d)$ and $j\geq -1$, the Littlewood-Paley blocks of $u$ are defined as
\begin{equation*}
    \Delta_j u = \mathscr{F}^{-1}_{\mathbb{T}^d}(\varrho_j \mathscr{F}_{\mathbb{T}^d}u),
\end{equation*}
where $\varrho_{-1}=\chi$ and $\varrho_j=\varrho(2^{-j}\cdot)$ for $j\geq 0$.
\end{defn}

\begin{defn}
For $\alpha\in\mathbb{R}$, $p,q \in [1,\infty]$, we define
\begin{equation*}
    B^{\alpha}_{p,q}(\mathbb{T}^d) = \left\{ u\in \mathcal{S}'(\mathbb{T}^d): \|u\|_{B^{\alpha}_{p,q}(\mathbb{T}^d)}=  	\left( \sum_{j \geq -1}(2^{j\alpha}\|\Delta_j u\|_{L^p(\mathbb{T}^d)})^q \right)^{1/q} < \infty \right\}.
\end{equation*}
\end{defn}

For $\alpha\in\mathbb{R}$, the H\"{o}lder-Besov space on $\mathbb{T}^d$ is denoted by $\mathscr{C}^{\alpha}=B^{\alpha}_{\infty,\infty}(\mathbb{T}^d)$. We remark that if $\alpha \in (0,\infty)\backslash\mathbb{N}$, then the H\"{o}lder-Besov space $\mathscr{C}^{\alpha}$ is equal to the H\"{o}lder space $C^{\alpha}(\mathbb{T}^d)$. The Sobolev space $H^{\alpha}$ is the same as the Besov space $B^{\alpha}_{2,2}(\mathbb{T}^d)$.

For a time weight $\eta$, we write
$C^{\beta}_{\eta}\mathscr{C}^{\alpha}$ for the space of continuous maps $\mathbb{R}^{+}\rightarrow \mathscr{C}^{\alpha}$ with norm $\|f\|_{C_{\eta}\mathscr{C}^{\alpha}} = \sup_{t\geq 0}\|\eta(t)f(t)\|_{\mathscr{C}^{\alpha}}$. For $\beta\in (0,1)$, we also denote
\begin{equation*}
    C^{\beta}_{\eta}\mathscr{C}^{\alpha} = \{ f\in C_{\eta}\mathscr{C}^{\alpha} : \|f\|_{C^{\beta}_{\eta}\mathscr{C}^{\alpha}}= \|f\|_{C_{\eta}\mathscr{C}^{\alpha}}+\sup_{t>s\geq0}\frac{\|\eta(t)f(t)-\eta(s) f(s)\|_{\mathscr{C}^{\alpha}}}{|t-s|^{\beta}}< \infty \}.
\end{equation*}

The following Bernstein inequality is useful in our estimates.
\begin{lem}
Let $\mathscr{B}$ ba a ball, $n \in \mathbb{N}_0$, and $1\leq p \leq q \leq \infty$. Then for every $\lambda >0$ and $u \in L^p$ with $supp(\mathscr{F}u) \subset \lambda \mathscr{B}$, we have
\begin{equation*}
    \max_{\mu \in \mathbb{N}^d: |\mu|=n}\|\partial_{\mu} u\|_{L^q} \lesssim C_{n,p,q,\mathscr{B}}\lambda^{n+d(\frac{1}{p}-\frac{1}{q})}\|u\|_{L^p}.
\end{equation*}
\end{lem}

We need the following Bernstein inequality in $L^2$ estimates.
\begin{lem}\label{BerI}
Let $\mathscr{B}$ ba a unit ball, $n \in \mathbb{N}_0$, and $1\leq p \leq q \leq \infty$. Then for every $\lambda >0$ and $u \in L^p$ with $supp(\mathscr{F}u) \subset \lambda \mathscr{B}$, we have
\begin{equation*}
    \max_{\mu \in \mathbb{N}^d: |\mu|=n}\|\partial_{\mu} u\|_{L^q} \lesssim C_{n,p,q,\mathscr{B}}\lambda^{n+d(\frac{1}{p}-\frac{1}{q})}\|u\|_{L^p}.
\end{equation*}
\end{lem}

The Besov embedding theorem is useful in regularity estimates.
\begin{lem}\label{Besovem}
Let $1 \leq p_1 \leq p_2 \leq \infty$, $1 \leq q_1 \leq q_2 \leq \infty $, and $\alpha \in\mathbb{R}$. Then we have
\begin{equation*}
    B^{\alpha}_{p_1, q_1}(\mathbb{T}^d) \hookrightarrow B^{\alpha -d(1/p_1 -1/p_2)}_{p_2, p_2}(\mathbb{T}^d).
\end{equation*}
\end{lem}

Now  we  define  localization  operators $\mathscr{U}^{N,\gamma}_{\leq} $, $\mathscr{U}^{N,\gamma}_{>}$ for the high-low frequency decomposition. For every $f \in \mathcal{S}^{\prime}(\mathbb{T}^d)$, we define the following localization operators
\begin{equation}
    \mathscr{U}^{N,\gamma}_{\leq} f= \sum_{-1 \leq j \leq N}\Delta_{j} f + \sum_{j > N} 2^{- j \gamma} \Delta_{j} f, \quad  \mathscr{U}^{N,\gamma}_{>} f=  \sum_{j > N}(1-2^{- j \gamma})\Delta_{j} f.
\end{equation}

\begin{lem}\label{Localization}
Let $N, \gamma>0$ and $f \in \mathcal{S}^{\prime}(\mathbb{T}^d)$. Then for every $\alpha, \delta>0$ and $\beta \in [0,\gamma],$ we have
\begin{equation*}
   \left\|\mathscr{U}^{N,\gamma}_{>} f\right\|_{\mathscr{C}^{-\alpha-\delta}} \lesssim 2^{-\delta N}\|f\|_{\mathscr{C}^{-\alpha} }, \quad \|\mathscr{U}^{N,\gamma}_{\leq} f\|_{\mathscr{C}^{-\alpha+\beta}} \lesssim 2^{\beta N}\|f\|_{\mathscr{C}^{-\alpha} }.
\end{equation*}
\end{lem}

\begin{proof}
We estimate
\begin{align}
    \|\mathscr{U}^{N,\gamma}_{>} f \|_{\mathscr{C}^{-\alpha-\delta}} = & \sup_{l \geq -1} \left[ 2^{l(-\alpha - \delta)} \|\Delta_{l} (\sum_{j > N}(1-2^{- j \gamma})\Delta_{j} f) \|_{L^{\infty}}\right] \nonumber\\
    \leq & 2^{-\delta N}\sup_{l \geq -1} \left[ 2^{-\alpha l} \|\Delta_{l} f \|_{L^{\infty}}\right] \nonumber\\
    \leq & 2^{-\delta N} \|f\|_{\mathscr{C}^{-\alpha} }
\end{align}
Using same argument, we also have
\begin{align}
    \|\mathscr{U}^{N,\gamma}_{\leq} f\|_{\mathscr{C}^{-\alpha+\beta}} = & \sup_{l \geq -1} \left[ 2^{l(-\alpha + \beta)} \|\Delta_{l}(\sum_{-1 \leq j \leq N}\Delta_{j} f + \sum_{j \geq N} 2^{- j \gamma} \Delta_{j} f ) \|_{L^{\infty}}\right] \nonumber\\
     \leq & 2^{\beta N} \|f\|_{\mathscr{C}^{-\alpha} }.
\end{align}

\end{proof}

Now we introduce the Bony's paraproduct. Let $u$ and $v$ be tempered distributions in $\mathcal{S}^{\prime}(\mathbb{T}^d)$. By Littlewood-Paley blocks, the product $uv$ can be (formally) decomposed as
\begin{equation*}
    uv = \sum_{j\geq -1}\sum_{i\geq -1} \Delta_{i} u \Delta_{j} v = u\prec v +u\circ v+ u\succ v,
\end{equation*}
where
\begin{equation*}
    u\prec v = v\succ u= \sum_{j\geq -1}\sum_{i =-1}^{j-2} \Delta_{i} u \Delta_{j} v \quad \text{and} \quad u\circ v = \sum_{|i-j|\leq 1} \Delta_{i} u\Delta_{j} v.
\end{equation*}
We have following paraproduct estimates in the Bony's paraproduct (See Lemma 2.1 in [\refcite{GIP2015}] or Proposition A.1 in [\refcite{GUZ2020}]).
\begin{lem}\label{Bparaproduct}
For every $\beta \in \mathbb{R}$, we have
\begin{equation*}
    \|u\prec v\|_{\mathscr{C}^{\beta}} \lesssim \|u\|_{L^{\infty}}\|v\|_{\mathscr{C}^{\beta}},
\end{equation*}
\begin{equation*}
    \|u \prec v\|_{H^{\beta}} \lesssim \|u\|_{L^2} \|v\|_{\mathscr{C}^{\beta + \kappa }} \wedge \|u\|_{L^{\infty}}\|v\|_{H^{\beta}} \quad \text{for all }\kappa >0.
\end{equation*}
If $\beta \in \mathbb{R}$, $\alpha <0$, we have
\begin{equation*}
    \|u\prec v\|_{\mathscr{C}^{\alpha+\beta}} \lesssim \|u\|_{\mathscr{C}^{\alpha}}\|v\|_{\mathscr{C}^{\beta}},
\end{equation*}
\begin{equation*}
    \|u \prec v\|_{H^{\alpha + \beta }} \lesssim \|u\|_{H^{\alpha}}\|v\|_{\mathscr{C}^{\beta + \kappa }} \wedge \|u\|_{\mathscr{C}^{\alpha}}\|v\|_{H^{\beta}} \quad \text{for all }\kappa >0.
\end{equation*}
Moreover, if $\alpha+\beta>0$, then
\begin{equation*}
\|u\circ v\|_{\mathscr{C}^{\alpha +\beta}} \lesssim \|u\|_{\mathscr{C}^{\alpha}}\|v\|_{\mathscr{C}^{\beta}},
\end{equation*}
\begin{equation*}
    \|u\circ v\|_{H^{\alpha + \beta}} \lesssim \|u\|_{\mathscr{C}^{\alpha}}\|v\|_{H^{\beta}}.
\end{equation*}
\end{lem}

The following commutator estimate is also crucial in paracontrolled distribution (See Lemma 2.4 in [\refcite{GIP2015}] and Proposition A.2 in [\refcite{GUZ2020}])

\begin{lem}\label{commutatorE}
. Assume that $\alpha\in (0,1)$ and $\beta, \gamma \in \mathbb{R}$ are such that $\alpha+\beta +\gamma>0$ and $\beta +\gamma <0$. Then for $u,v,h \in C^{\infty}(\mathbb{T}^d)$, the trilinear operator
\begin{equation*}
    C(u,v,h)=(u\prec v)\circ h-u(v\circ h)
\end{equation*}
has the following estimate
\begin{equation*}
    \|C(u,v,h)\|_{\mathscr{C}^{\alpha+\beta +\gamma}} \lesssim \|u\|_{\mathscr{C}^{\alpha}}\|v\|_{\mathscr{C}^{\beta}}\|h\|_{\mathscr{C}^{\gamma}}.
\end{equation*}
Thus $C$ can be uniquely extended to a bounded trilinear operator from $\mathscr{C}^{\alpha}\times \mathscr{C}^{\beta} \times \mathscr{C}^{\gamma}$ to $\mathscr{C}^{\alpha+\beta +\gamma}$. For $H^{\alpha}$ space, we also have
\begin{equation*}
    \|C(u,v,h)\|_{H^{\alpha + \beta +\gamma}} \lesssim \|u\|_{ H^{\alpha}}\|v\|_{H^{\beta}}\|h\|_{\mathscr{C}^{\gamma }}.
\end{equation*}
It implies that $C$ can be uniquely extended to a bounded trilinear operator from $H^{\alpha}\times H^{\beta}\times \mathscr{C}^{\gamma}$ to $H^{\alpha+\beta +\gamma}$.
\end{lem}

For every $u,v,h \in C^{\infty}(\mathbb{T}^d)$, we define the trilinear operator
\begin{equation}
    D(u,v,h) = \langle u, h\circ v \rangle - \langle u\prec v, h\rangle.
\end{equation}
We have the following estimate from [\refcite{GUZ2020}].
\begin{lem}\label{Dm}
Let $\alpha\in (0,1)$, $\beta, \gamma \in \mathbb{R}$ such that $\alpha + \beta + \gamma > 0$ and $\beta + \gamma < 0$. Then we have
\begin{equation*}
    |D(u,v,h)| \lesssim \|u\|_{ H^{\alpha}}\|v\|_{H^{\beta}}\|h\|_{\mathscr{C}^{\gamma}}.
\end{equation*}
Thus $D$ can be uniquely extended to a bounded trilinear operator from $H^{\alpha}\times H^{\beta}\times \mathscr{C}^{\gamma}$ to $\mathbb{R}$.
\end{lem}

The following estimate from [\refcite{AC2015}] is useful in this chapter.
\begin{lem}\label{CE}
Let $f\in H^{\alpha}$, $g \in \mathscr{C}^{\beta}$ with $\alpha \in (0,1)$, $\beta \in \mathbb{R}$. Then
\begin{equation*}
    \|\mathscr{L}(u\prec v)- u\prec (\mathscr{L}v)\|_{H^{\alpha+\beta+2}} \lesssim \|u\|_{H^{\alpha}}\|v\|_{\mathscr{C}^{\beta}}.
\end{equation*}
\end{lem}

In order to obtain some estimate uniformly in time, we also need the following time-mollified paraproducts from [\refcite{GIP2015}].

\begin{defn}
Let $\phi: \mathbb{R}\rightarrow \mathbb{R}^{+}$ be a smooth function with compact support $supp \phi \subset [-1,1]$, and $\int_{\mathbb{R}}\phi(s)ds =1$. Let $\eta$ be a time weight. For all $i \geq -1$, we define the operator $Q_i: C_{\eta}\mathscr{C}^{\alpha} \rightarrow C_{\eta}\mathscr{C}^{\alpha} $ by
\begin{equation*}
    Q_i u(t):= \int_{\mathbb{R}}2^{2i}\phi(2^{2i}(t-s))u(s\vee 0)\eta(s)ds.
\end{equation*}
And we define the modified paraproduct of $u,v \in C_{\eta}\mathscr{C}^{\alpha}$ by
\begin{equation}\label{modifiedpara}
    u \pprec v=\sum_{i}\left(\sum_{j=-1}^{i-1}\Delta_{j}(Q_{i} u)\right) \Delta_{i} v.
\end{equation}
\end{defn}
The following two estimates are the useful properties of $\pprec$ from Lemma 2.17 in [\refcite{GH2019}].

\begin{lem}\label{pprec}
Let $\alpha\in (0,1)$, $\beta \in \mathbb{R}$, and let $u\in  C\mathscr{C}^{\alpha}\cap C^{\alpha/2} L^{\infty}$ and $v\in C \mathscr{C}^{\beta}$. Then
\begin{equation*}
    \|\mathscr{L}(u\pprec v)-u\pprec (\mathscr{L}v)\|_{C\mathscr{C}^{\alpha+\beta-2}} \lesssim (\|u\|_{ C\mathscr{C}^{\alpha}}+ \|u\|_{C^{\alpha} L^{\infty}})\|v\|_{ C \mathscr{C}^{\beta}},
\end{equation*}
and
\begin{equation*}
    \|u\prec v -u\pprec v\|_{C\mathscr{C}^{\alpha+\beta}} \lesssim \|u\|_{C^{\alpha/2} L^{\infty}}\|v\|_{ C\mathscr{C}^{\beta}}.
\end{equation*}
\end{lem}

We will need the following interpolations result for Besov space.
\begin{lem}\label{interpolation}
Let $\eta$ be time weights, $\gamma>0$, $\theta \geq 0$, and $\psi \in C_{\eta}\mathscr{C}^{\gamma}$. Then for any $\alpha \in [0, \gamma]$, we have
\begin{equation}
    \|\psi\|_{C_{\eta^{1+\theta}}\mathscr{C}^{\alpha}} \lesssim \| \psi\|_{C_{\tau^{1/(k-2)}}L^{\infty}}^{\alpha/\gamma}  \| \psi\|_{C_{\eta^{1+\theta\gamma/\alpha}}\mathscr{C}^{\gamma}}^{1-\alpha/\gamma}.
\end{equation}
Moreover, if $\alpha \in (0,1)$ then
\begin{equation}
    \|\psi\|_{C^{\alpha/2}_{\eta}L^{\infty}} \lesssim \|\psi\|^{1/2}_{C^{\alpha}_{\eta}L^{\infty}}\|\psi\|^{1/2}_{C_{\tau^{1/(k-2)}}L^{\infty}}.
\end{equation}
\end{lem}
\begin{proof}
For spatial regularity, it holds
\begin{align*}
    \eta(t)^{1+\theta}\|\Delta_k \psi(t)\|_{L^{\infty}} \lesssim & \left[ \eta(t)\|\Delta_k \psi\|_{L^{\infty}}\right]^{1-\alpha/\gamma}\left[ \eta(t)^{1+\theta\gamma/\alpha}\|\Delta_k \psi\|_{L^{\infty}}\right]^{\alpha/\gamma} \\
     \lesssim & 2^{-\alpha k} \left[ \eta(t)\|\Delta_k \psi\|_{L^{\infty}}\right]^{1-\alpha/\gamma} \left[ \eta(t)^{1+\theta\gamma/\alpha}\| \psi\|_{\mathscr{C}^{\gamma}}\right]^{\alpha/\gamma}.
\end{align*}
Thus for each $t>0$, we have
\begin{equation*}
    \eta(t)^{1+\theta}\|\psi(t)\|_{\mathscr{C}^{\alpha}} \lesssim \left[ \eta(t)\| \psi\|_{L^{\infty}}\right]^{\alpha/\gamma} \left[ \eta(t)^{1+\theta\gamma/\alpha}\| \psi\|_{\mathscr{C}^{\gamma}}\right]^{1-\alpha/\gamma}.
\end{equation*}
Taking supremum in time, we obtain
\begin{equation*}
    \|\psi\|_{C_{\eta^{1+\theta}}\mathscr{C}^{\alpha}} \lesssim \| \psi\|_{C_{\tau}L^{\infty}}^{\alpha/\gamma}  \| \psi\|_{C_{\eta^{1+\theta\gamma/\alpha}}\mathscr{C}^{\gamma}}^{1-\alpha/\gamma}.
\end{equation*}
For time regularity, we have
\begin{align*}
    \|\psi\|_{C^{\alpha/2}_{\eta}L^{\infty}} = & \|\psi\|_{C_{\eta}L^{\infty}} + \sup_{t>s\geq0}\frac{\|\eta(t)\psi(t)-\eta(s) \psi(s)\|_{L^{\infty}}}{|t-s|^{\alpha/2}} \\
    \leq & \|\psi\|_{C_{\eta}L^{\infty}} + \sup_{t>s\geq0}\frac{\|\eta(t)\psi(t)-\eta(s) \psi(s)\|^{1/2}_{L^{\infty}}}{|t-s|^{\alpha/2}} \|\psi\|^{1/2}_{C_{\eta}L^{\infty}} \\
    \leq & \|\psi\|^{1/2}_{C^{\alpha}_{\eta}L^{\infty}}\|\psi\|^{1/2}_{C_{\eta}L^{\infty}}.
\end{align*}
This completes the proof.
\end{proof}

\begin{lem}\label{interpolationH}
Let $\beta \in (0,1)$ and $\psi \in H^{\beta}$. Then  for arbitrary $\delta>0$, we have
\begin{equation}
    \|\psi\|^2_{H^{\beta}} \lesssim \delta \| \nabla \psi \|_{L^2}^2 + C_{\delta} \|\psi \|_{L^2}^2.
\end{equation}
\end{lem}
\begin{proof}
Since $\|\psi\|_{H^{\beta}} \simeq \|\psi\|_{B_{2,2}^{\beta}}$, by Bernstein inequality (Lemma \ref{BerI}), H\"{o}lder inequality and weighted Young inequality, we have
\begin{align}
    \|\psi\|^2_{H^{\beta}} & = \sum_{i\geq -1} 2^{2 \beta k } \|\Delta_{i} \psi\|^2_{L^2} \nonumber \\
    & = \sum_{i\geq -1} 2^{2 \beta k } \|\Delta_{i} \psi\|^{2\beta}_{L^2} \|\Delta_{i} \psi\|^{2(1-\beta)}_{L^2} \nonumber \\
    & \leq \left[ \sum_{i\geq -1}  2^{2  k }\|\Delta_{i}\psi\|^{2}_{L^2} \right]^{2\beta} \left[ \sum_{i\geq -1}  \|\Delta_{i}\psi\|^{2}_{L^2} \right]^{2(1-\beta)} \nonumber \\
    & \lesssim\| \nabla \psi \|^{\beta} \|\psi\|^{1-\beta}_{L^2}  \nonumber \\
    & \lesssim \delta \| \nabla \psi \|_{L^2}^2 + C_{\delta} \|\psi \|_{L^2}^2.
\end{align}
This completes the proof.
\end{proof}

\subsection{Renormalization and paracontrolled distributions}

The spatial white noise $\xi$ on $\mathbb{T}^2$ is a centered Gaussian process with value in $\mathcal{S}'(\mathbb{T}^2)$ such that for all $f,g\in \mathcal{S}(\mathbb{T}^2)$, we have $\mathbb{E}[\xi(f)\xi(g)]=\langle f,g\rangle_{L^2 (\mathbb{T}^2)}$. Let $(\hat{\xi}(k))_{k\in \mathbb{Z}^2}$ be a sequence of i.i.d. centered complex Gaussian random variables with covariance
\begin{equation*}
    \mathbb{E}(\hat{\xi}(k)\Bar{\hat{\xi}}(l)) = \delta(k-l),
\end{equation*}
and $\hat{\xi}(k) = \Bar{\hat{\xi}}(-k)$. Then the spatial white noise $\xi$ on $\mathbb{T}^2$ can be defined as follows
\begin{equation*}
    \xi(x) =\sum_{k \in \mathbb{Z}^2} \hat{\xi}(k)e^{2\pi ik \cdot x}.
\end{equation*}

Moreover, the  spatial white noise $\xi$ take value in $\mathscr{C}^{-1-\kappa}$ for all $\kappa>0$. Since $\xi$ is only a distribution, $u\xi$ is ill-defined in classic sense. How to let singular term $u\xi$ make sense is a main challenge in studying the parabolic Anderson mode equation. It is natural to replace $\xi$ by a smooth approximation $\xi_{\epsilon}$ which is given by the convolution of $\xi$ with a rescaled mollifier $\varphi$. More precisely, we let $\varphi: \mathbb{T}^2 \rightarrow \mathbb{R}^{+}$ be a smooth function with $\int_{\mathbb{T}^2}\varphi dt =1$, and define $\xi^{\epsilon} = \epsilon^{-2}\varphi(\epsilon\cdot)\ast\xi$ for $\epsilon >0$ as the mollification of $\xi$.

For the PAM equation (\ref{GPAM}), we take 
\begin{equation*}
    \vartheta = (-\Delta + \mu)^{-1}\xi = \int_0^{\infty} e^{t(\Delta- \mu)} \xi dt,
\end{equation*}
where $(e^{t(\Delta - \mu)})_{t\geq 0}$ denotes the semigroup generated by $\Delta - \mu$. Then $\vartheta \in \mathscr{C}^{1-\kappa}$, and $\|\vartheta\|_{\mathscr{C}^{1-\kappa}} \lesssim \|\xi\|_{\mathscr{C}^{-1-\kappa}}$.
In order to obtain a well-defined area $\vartheta \diamond \xi$, we have to renormalize the product by “subtracting an infinite constant” as following arguments (see Lemma 5.8 in \refcite{GIP2015}).

\begin{lem}\label{renormalize}
If $\vartheta_{\epsilon} = (-\Delta + \mu)^{-1}\xi_{\epsilon}$, then the wick product $\vartheta \diamond \xi$ can be approximated as
\begin{equation*}
    \lim_{\epsilon\rightarrow 0}\mathbb{E}[\|\vartheta\diamond\xi-(\vartheta_{\epsilon}\circ\xi_{\epsilon}-C_{\epsilon})\|^p_{\mathscr{C}^{-2\kappa}}]=0
\end{equation*}
for all $p\geq 1$ and $\kappa >0$ with the renormalization constant
\begin{equation*}
    C_{\epsilon}= \mathbb{E}(\vartheta_{\epsilon}\circ\xi_{\epsilon})  = \sum_{k\in \mathbb{Z}^2} \frac{|\mathscr{F}_{\mathbb{T}^2}\varphi(\epsilon k)|^2}{|k|^2+ \mu} .
\end{equation*}
\end{lem}

Using the modified paraproduct $\pprec$, we introduce paracontrolled distributions as follows.
\begin{defn}
Let $\alpha\in (2/3,1)$ and $\beta \in (0,\alpha]$ be such that $2\alpha+\beta >2$. Let $\rho'$ be a time weight. We say a pair $(u, u')\in { C_{\rho'}\mathscr{C}^{\alpha}}\times{ C_{\rho'}\mathscr{C}^{\beta}}$ is called paracontrolled by $\vartheta$ if 
\begin{equation*}
    u^{\sharp}:= u-u'\pprec \vartheta \in { C_{\rho'}\mathscr{C}^{\alpha+\beta}}.
\end{equation*}
\end{defn}

Now we define $u\diamond\xi$ by above the renormalization argument of singular term $\vartheta\diamond\xi$ and paracontrolled distributions. If $u\in C_{\rho}\mathscr{C}^{\alpha}$ is paracontrolled by $\vartheta$: $u^{\sharp}:= u- u\pprec\vartheta \in C_{\rho}\mathscr{C}^{2\alpha} $, then we define $u\diamond\xi$ as following
\begin{align*}
    u\diamond\xi= & u\prec \xi + u\succ \xi+ u\circ\xi \\
    = & u\prec \xi + u\succ \xi +(u\pprec\vartheta)\circ\xi + u^{\sharp}\circ\xi  \\
              = & u\prec \xi + u\succ \xi + (u\pprec\vartheta - u\prec \vartheta)\circ\xi + C(u,\vartheta,\xi) +u (\vartheta \diamond \xi) +u^{\sharp}\circ\xi \\
               = & \lim_{\epsilon\rightarrow 0}(u\prec \xi_{\epsilon} + u\succ \xi_{\epsilon} + (u\pprec\vartheta_{\epsilon} - u\prec \vartheta_{\epsilon})\circ\xi_{\epsilon} + C(u,\vartheta_{\epsilon},\xi_{\epsilon}) + u (\vartheta_{\epsilon} \circ \xi_{\epsilon}-C_{\epsilon}) +u^{\sharp}\circ\xi_{\epsilon}). 
\end{align*}
Thus the singular term $u\diamond\xi$ can be formally written as 
\begin{equation*}
    u\diamond\xi =\lim_{\epsilon\rightarrow 0}u\xi_{\epsilon}-C_{\epsilon}u= u\xi - \infty\cdot u.
\end{equation*}

\subsection{Parabolic Schauder estimates}
We recall the following Schauder estimate for the heat semigroup $P_t:=e^{t(\Delta-\mu)}$ from [\refcite{GIP2015}].

\begin{lem}\label{Lsch}
Let $\alpha\in \mathbb{R}$, $\beta \in [0,2]$, and let $P_t$ be the semigroup generated by $\Delta-\mu$ with $\mu>0$. Then for every $t\geq 0$, $u_0 \in \mathscr{C}^{\alpha-\beta}$, we have 
\begin{equation*}
    \|P_t u_0\|_{\mathscr{C^{\alpha}}} \lesssim e^{-\mu t}t^{-\beta/2}\|u_0\|_{\mathscr{C}^{\alpha-\beta}}.
\end{equation*}
\end{lem}

For time weight $\tau(t):=1-e^{-t}$, we have the following Schauder estimates from [\refcite{GH2019}].

\begin{lem}\label{tauSch}
Define $ \tau(t):=1-e^{-t}$ be a time weight. Let $\alpha\in \mathbb{R}$ and $\beta,\beta_i \in [0,2)$. Assume that $v \in C_{[0,\infty)}\mathscr{C}^{\alpha}$ with $v(0)=0$ be a solution of 
\begin{equation*}
    \mathscr{L}v = \sum_{i}f_i.
\end{equation*}
Then we have
\begin{equation*}
    \|v\|_{C_{[0,\infty)}\mathscr{C}^{\alpha}} \lesssim \|\tau^{\beta/2}v\|_{C_{[0,\infty)}\mathscr{C}^{\alpha+\beta-2}} + \sum_{i}\|\tau^{\beta_1/2}f_i \|_{C_{[0,\infty)}\mathscr{C}^{\alpha+\beta_i -2}}.
\end{equation*}
Moreover, for every $\alpha \in (0,2)$ and $\beta_i \in [0,2)$ such that $\alpha + \beta_i -2 <0$, we have
\begin{equation*}
    \|v\|_{C_{[0,\infty)}^{\alpha/2}L^{\infty}} \lesssim \|v\|_{C_{[0,\infty)}\mathscr{C}^{\alpha}} + \sum_{i}\|\tau^{\beta_1/2}f_i \|_{C_{[0,\infty)}\mathscr{C}^{\alpha+\beta_i -2}}.
\end{equation*}
\end{lem}

We also need the following Schauder estimate for parabolic equations with polynomial nonlinear term.
\begin{lem}\label{nlSch}
Define $ \tau(t):=1-e^{-t}$ be a time weight. Let $\mu>0$, $\beta \in [0,1)$, and $\Psi \in C_{\tau^{1+1/(k-1)+\beta/2}}\mathscr{C}^{\beta}$. Assume that 
\begin{equation*}
    \psi \in C_{\tau^{1+1/(k-1)+\beta/2}}\mathscr{C}^{2+\beta}\cap C_{\tau^{1+1/(k-1)+\beta/2}}\mathscr{C}^{\beta} \cap C_{\tau^{1/(k-2)}}L^{\infty}
\end{equation*} 
be a classical solution to
\begin{equation*}
    \mathscr{L}\psi = f(\psi) + \Psi, \quad \psi(0)= 0.
\end{equation*}
If $f$ satisfies the dissipative assumption (\ref{assumf}), then
\begin{align}
    &\| \psi\|_{C_{\tau^{1+1/(k-1)+\beta/2}}\mathscr{C}^{2+\beta}} + \| \psi\|_{C^1_{\tau^{1+1/(k-1)+\beta/2}}L^{\infty}} \nonumber \\
    \lesssim & \|\psi\|_{C_{\tau^{1+1/(k-1)+\beta/2}}\mathscr{C}^{\beta}} + \|\Psi\|_{C_{\tau^{1+1/(k-1)+\beta/2}}\mathscr{C}^{\beta}} + 1+ \left( \|\psi\|_{C_{\tau^{1/(k-2)}}L^{\infty}}^{k-2+ 2/(2+\beta)} \right)^{\frac{2+\beta}{2}}.    
\end{align}
\end{lem}

\begin{proof}
By Lemma \ref{tauSch}, we have
\begin{align*}
    & \| \psi\|_{C_{\tau^{1+1/(k-1)+\beta/2}}\mathscr{C}^{2+\beta}} \\
    \lesssim & \|\psi\|_{C_{\tau^{1+1/(k-1)+\beta/2}}\mathscr{C}^{\beta}} + \|\Psi\|_{C_{\tau^{1+1/(k-1)+\beta/2}}\mathscr{C}^{\beta}} + \|f(\psi)\|_{C_{\tau^{1+1/(k-1)+\beta/2}}\mathscr{C}^{\beta}}.    
\end{align*}
The interpolation result in Lemma \ref{interpolation} and the weighted Young inequality lead to
\begin{align*}
  &\|f(\psi)\|_{C_{\tau^{1+1/(k-1)+\beta/2}}\mathscr{C}^{\beta}} \\
   \lesssim & 1 + \|\psi\|_{C_{\tau^{1/(k-2)}}L^{\infty}}^{k-2} \|\psi\|_{C_{\rho^{1+(k-2)\beta/2}}\mathscr{C}^{\beta}} \\
    \lesssim &  1 + \|\psi\|_{C_{\tau^{1/(k-2)}}L^{\infty}}^{k-2} \|\psi\|_{C_{\tau^{1/(k-2)}}L^{\infty}}^{2/(2+\beta)} \|\psi\|_{C_{\tau^{1+1/(k-1)+\beta/2}}\mathscr{C}^{2+\beta}}^{\beta/(2+\beta)} \\
    \lesssim  & 1+  C_{\lambda}\left( \|\psi\|_{C_{\tau^{1/(k-2)}}L^{\infty}}^{k-2+ 2/(2+\beta)} \right)^{\frac{2+\beta}{2}} + \lambda  \|\psi\|_{C_{\tau^{1+1/(k-1)+\beta/2}}\mathscr{C}^{2+\beta}} 
\end{align*}
for arbitrary $\lambda>0$. Choosing $\lambda$ small enough, we have
\begin{align*}
        \| \psi\|_{C_{\tau^{1+1/(k-1)+\beta/2}}\mathscr{C}^{2+\beta}}
    \lesssim & \|\psi\|_{C_{\tau^{1+1/(k-1)+\beta/2}}\mathscr{C}^{\beta}} + \|\Psi\|_{C_{\tau^{1+1/(k-1)+\beta/2}}\mathscr{C}^{\beta}} \\
    & + 1+ \left( \|\psi\|_{C_{\tau^{1/(k-2)}}L^{\infty}}^{k-2+ 2/(2+\beta)} \right)^{\frac{2+\beta}{2}}.
\end{align*}
Using Lemma \ref{tauSch}, we obtain the time regularity. The proof is completed.
\end{proof}

We also need the following parabolic coercive estimates.

\begin{lem}\label{coercive}
Define $ \tau(t):=1-e^{-t}$ be a time weight. Let $\mu>0$, $\beta \in [0,1)$, and $\Psi \in C_{\tau^{1+1/(k-2)}}L^{\infty}$. Assume that 
\begin{equation*}
    \psi \in C_{\tau^{1+1/(k-1)+\beta/2}}\mathscr{C}^{2+\beta}\cap C_{\tau^{1+1/(k-1)+\beta/2}}\mathscr{C}^{\beta} \cap C_{\tau^{1/(k-2)}}L^{\infty}
\end{equation*} 
be a classical solution to
\begin{equation*}
    \mathscr{L}\psi = f(\psi) + \Psi, \quad \psi(0)= 0.
\end{equation*}
If $f$ satisfies the dissipative assumption (\ref{assumf}), then we have a priori estimates 
\begin{equation*}
    \|\psi\|_{C_{\tau^{1/(k-2)}}L^{\infty}} \lesssim 1+ \|\Psi\|_{C_{\rho^{1+1/(k-2)}}L^{\infty}}^{1/(k-1)}.
\end{equation*}
\end{lem}

\begin{proof}
Let $\hat{\psi}(t,x)=\psi(t,x)\tau(t)^{\frac{1}{k-2}}$. Suppose $\hat{\psi}$ attains its global maximum $M$ at $(t^{\ast},x^{\ast}) \in [0,\infty)\times\mathbb{T}^2 $. We first assume that $M>0$. Since $\hat{\psi}(0)=\psi(0)\tau(0)^{\frac{1}{k-2}}=0$, $t^{\ast}>0$, and we have
\begin{equation*}
    \partial_t \hat{\psi}(t^{\ast},x^{\ast})=0, \quad -\Delta\hat{\psi}(t^{\ast},x^{\ast})\geq 0.
\end{equation*}
Furthermore, $\hat{\psi}$ satisfies
\begin{align*}
    & \partial_t \hat{\psi}(t^{\ast},x^{\ast}) + (-\Delta+\mu)\hat{\psi}(t^{\ast},x^{\ast}) \\
    = & F(\psi(t^{\ast},x^{\ast}))\tau(t^{\ast})^{\frac{1}{k-2}} + \Psi(t^{\ast},x^{\ast})\tau(t^{\ast})^{\frac{1}{k-2}} + \psi(t^{\ast},x^{\ast})\partial_t\tau(t^{\ast})^{\frac{1}{k-2}},
\end{align*}
Then by assumption (\ref{assumf}), we have
\begin{align*}
    \mu M\tau(t^{\ast}) + M^{k-1} & \leq \Psi(t^{\ast},x^{\ast})\tau^{1+1/(k-2)}(t^{\ast}) + \frac{1}{k-2}(\tau\partial_t \tau)(t^{\ast})M \\
    & \leq \|\Psi\|_{C_{\tau^{1+1/(k-2)}}L^{\infty}}^{1/(k-1)} + \frac{1}{k-2}(\tau\partial_t \tau)(t^{\ast})M.
\end{align*}
Since $k\geq 3$, the weighted term $\frac{1}{k-2}(\tau\partial_t \tau)$ is bounded. Then we conclude that
\begin{equation*}
    \hat{\psi}(t^{\ast},x^{\ast}) \lesssim \|\psi\|^{1/(k-1)}_{C_{\tau^{1/(k-2)}}L^{\infty}(\mathbb{T}^2)}+\|\Psi\|_{C_{\tau^{1+1/(k-2)}}L^{\infty}}^{1/(k-1)}.
\end{equation*}
Applying same argument to $-\hat{\psi}$, we also have
\begin{equation*}
    -\hat{\psi}(t^{\ast},x^{\ast}) \lesssim \|\psi\|^{1/(k-1)}_{C_{\tau^{1/(k-2)}}L^{\infty}(\mathbb{T}^2)}+\|\Psi\|_{C_{\tau^{1+1/(k-2)}}L^{\infty}}^{1/(k-1)}.
\end{equation*}
Then by weighted Young inequality, we get
\begin{equation*}
    \|\psi\|_{C_{\tau^{1/(k-2)}}L^{\infty}} \lesssim 1 +\|\Psi\|_{C_{\tau^{1+1/(k-2)}}L^{\infty}}^{1/(k-1)}.
\end{equation*}
If $\hat{\psi}=\psi\tau^{1/(k-2)}$ does not attain its global maximum at finite time, then for all $t_0 >0$ it holds that $\hat{\psi}(t_0) < \lim_{t\rightarrow \infty}\hat{\psi}(t)$.
Since $\hat{\psi}$ is bounded and continuous on $[0,\infty)\times \mathbb{T}^2$, then for every $\delta >0$ we have
\begin{equation*}
    \lim_{t \rightarrow \infty}\psi(t)\tau(t)^{1/(k-2)}(1+|t|^2)^{-\delta} = 0
\end{equation*}
Thus $\psi(t)\tau(t)^{1/(k-2)}(1+|t|^2)^{-\delta}$ attain its global maximum at finite time. Now we can use same argument in above proof to $\psi(t)\tau(t)^{1/(k-2)}(1+|t|^2)^{-\delta}$, and the conclusion follows by letting $\delta \rightarrow 0$. 
\end{proof}

Similar with Propostion A.2 in [\refcite{GH2019}], we have the following existence result.

\begin{lem}\label{Exapp}
Let $T>0$, $\mu>0$, $u_0 \in C^{\infty}(\mathbb{T}^2)$, $\xi_{\epsilon}\in C^{\infty}(\mathbb{T}^2)$ is the mollification of spatial white noise $\xi$. 
Then there exists a unique classical solution $u\in C^{\infty}(\mathbb{R}^{+}\times\mathbb{T}^2)$ to 
\begin{equation}\label{appPAM}
    \mathscr{L}u = F(u)+u\diamond\xi_{\epsilon}, \quad u(0)=u_0.
\end{equation}
\end{lem}

\begin{proof}
Note that $u_0 \in L^2(\mathbb{T}^2)$ and $\xi^{\epsilon}\in C^{\infty}(\mathbb{T}^2)$ is bounded in $[0,T]\times\mathbb{T}^2$. By monotonicity arguments with the Gelfand triplet(see e.g. Theorem 5.8 in [\refcite{J2001}]) 
\begin{equation*}
    [H^1(\mathbb{T}^2)\cap L^k(\mathbb{T}^2)] \hookrightarrow L^2(\mathbb{T}^2) \hookrightarrow [H^1(\mathbb{T}^2)\cap L^k(\mathbb{T}^2)]^{\ast},
\end{equation*}
equation (\ref{appPAM}) has a unique solution $u \in C_T L^2(\mathbb{T}^2)\cap L^2_T H^1(\mathbb{T}^2) \cap L^p_T L^k({\mathbb{T}^2})$. By Sobolev embedding $H^1(\mathbb{T}^2)\hookrightarrow L^p(\mathbb{T}^2)$ for every $p\in [2,\infty)$, we have $u(t)\in L^p(\mathbb{T}^2)$ for every $p\in [2,\infty)$ and $t\in [0,T]$ almost surely. After multiply $u^{p-1}$, $p\in [2,\infty)$ on both sides of the equation (\ref{appPAM}), we obtain
\begin{align*}
    & \frac{1}{p}\partial_t \int_{\mathbb{T}^2} |u|^{p}dx + (p-1)\int_{\mathbb{T}^2} |u|^{p-2}|\nabla u|^2dx  - \int_{\mathbb{T}^2} (c_0|u|^{p-2}-c_2|u|^{p+k-2})dx \\
    \leq & |\xi_{\epsilon}-C_{\epsilon}|_{ L^{\infty}(\mathbb{T}^2)}\int_{\mathbb{T}^2}  |u|^{p}dx- \mu\int_{\mathbb{T}^2} |u|^{p}dx.
\end{align*}
Then the Gronwall Lemma implies that $u\in L^{\infty}_{T}L^p(\mathbb{T}^2)$ for every $p\in [2,\infty)$. Thus by assumption (\ref{assumf}) and interpolation, we have that $F(u)+ u\diamond\xi_{\epsilon}\in L^{\infty}_{T}L^p(\mathbb{T}^2)$ for every $p\in [1,\infty)$. Applying a classical regularity result (see e.g. Theorem 3.2 in [\refcite{DdMH2015}]), we obtain that there exists $\alpha \in (0,1)$ and $p\in [1,\infty)$ such that
\begin{equation*}
    \|u\|_{C^{\alpha/2,\alpha}([0,T]\times\mathbb{T}^2)} \lesssim \|u_0\|_{\mathscr{C}^{\alpha}} +  (\mu+|\xi_{\epsilon}|_{L^{\infty}(\mathbb{T}^2)})\|u\|_{L^{\infty}_{T}L^p(\mathbb{T}^2)}+ \|F(u)\|_{L^{p}_{T}L^p(\mathbb{T}^2)}.
\end{equation*}
Moreover, since $F(u)\in C^{\alpha/2,\alpha}([0,T]\times\mathbb{T}^2)$, by Schauder estimates (see e.g. Theorem 3.4 in [\refcite{DdMH2015}]), we have
\begin{equation*}
    \|u\|_{C^{(\alpha+2)/2,\alpha+2}([0,T]\times\mathbb{T}^2)} \lesssim \|u_0\|_{\mathscr{C}^{\alpha+2}} +  (\mu+|\xi_{\epsilon}|_{L^{\infty}(\mathbb{T}^2)})\|u\|_{L^{\infty}_{T}L^p(\mathbb{T}^2)}+ \|F(u)\|_{C^{\alpha/2,\alpha}([0,T]\times\mathbb{T}^2)}.
\end{equation*}
Since $\xi^{\epsilon}\in C^{\infty}([0,T]\times\mathbb{T}^2)$ and $u_0 \in C^{\infty}(\mathbb{T}^2)$, we apply the regularity result 
 from Theorem 3.4 in [\refcite{DdMH2015}] repeatedly, and conclude that $u\in C^{\infty}([0,T]\times\mathbb{T}^2)$. Now we applied same argument as in parabolic coercive estimates Lemma \ref{coercive}, and then sending $T\rightarrow \infty$. This completes the proof.
\end{proof}

\section{Global well-posedness}

In this section, we consider the global existence and uniqueness of the following nonlinear parabolic Anderson model equation
\begin{equation*}
    \partial_t u + \mathscr{L}u = f(u)+ u \diamond \xi, \quad u(0)=u_0,
\end{equation*}
where $f$ is a continuous function from $\mathbb{R}$ to $\mathbb{R}$, and $\xi$ is a spatial white noise on the $2$-dimension torus $\mathbb{T}^2=(\mathbb{R}/\mathbb{Z})^2$.

Now we define $u\diamond\xi$ by above the renormalization argument of singular term $\vartheta\diamond\xi$ and paracontrolled distributions. If $u\in C_{\rho'}\mathscr{C}^{\alpha}$ is paracontrolled by $\vartheta$: $u^{\sharp}:= u- u\pprec\vartheta \in C_{\rho'}\mathscr{C}^{2\alpha} $ and define $u\diamond\xi$ as following
\begin{align*}
    & u\diamond\xi \\
    = & u\prec \xi + u\succ \xi+ u\circ\xi \\
    = & u\prec \xi + u\succ \xi +(u\pprec\vartheta)\circ\xi + u^{\sharp}\circ\xi  \\
              = & u\prec \xi + u\succ \xi + (u\pprec\vartheta - u\prec \vartheta)\circ\xi + C(u,\vartheta,\xi) +u (\vartheta \diamond \xi) +u^{\sharp}\circ\xi \\
               = & \lim_{\epsilon\rightarrow 0}(u\prec \xi_{\epsilon} + u\succ \xi_{\epsilon} + (u\pprec\vartheta_{\epsilon} - u\prec \vartheta_{\epsilon})\circ\xi_{\epsilon} + C(u,\vartheta_{\epsilon},\xi_{\epsilon}) + u (\vartheta_{\epsilon} \circ \xi_{\epsilon}-C_{\epsilon}) +u^{\sharp}\circ\xi_{\epsilon}).
\end{align*}
Thus the singular term $u\diamond\xi$ can be formally written as $u\diamond\xi =\lim_{\epsilon\rightarrow 0}u\xi_{\epsilon}-C_{\epsilon}u= u\xi - \infty\cdot u$.

We introduction the ansatz $u = \psi+\phi$. Then the original equation $\ref{GPAM}$ can be decomposed into a simple system
\begin{equation}\label{decGPAM}
  \left\{
   \begin{aligned}
   & \partial_t\phi + \mathscr{L}\phi = \Phi, \quad \phi(0) = \phi_0=u_0,\\
   & \partial_t\psi + \mathscr{L}\psi = f(\psi)+ \Psi, \quad \psi(0)=0,
   \end{aligned}
   \right.
\end{equation}
where $\Phi$ is the collection of all terms of negative regularity, and $\Psi$ the collection of all the others regular term (belonging to $L^{\infty}$).

Recall that the stochastic terms $\xi$ and $\vartheta\diamond\xi$ can be constructed such that
\begin{equation*}
    \|\xi\|_{\mathscr{C}^{-1-\kappa}} \lesssim 1, \quad \|\vartheta\diamond\xi\|_{\mathscr{C}^{-2\kappa}} \lesssim 1.
\end{equation*}
We choose small parameters $\kappa \in (0, 1- \alpha)$, and employ the Localization operators $\mathscr{U}_{\leq}$ and $\mathscr{U}_{>}$ to decompose
\begin{equation*}
    \xi =  \mathscr{U}_{\leq} \xi + \mathscr{U}_{>} \xi,  \quad  \vartheta\diamond\xi =  \mathscr{U}_{\leq} (\vartheta\diamond\xi) + \mathscr{U}_{>} (\vartheta\diamond\xi).
\end{equation*}
Here $\mathscr{U}_{\leq} \xi$, $\mathscr{U}_{\leq} (\vartheta\diamond\xi)$ are regular, and $\mathscr{U}_{>} \xi$, $\mathscr{U}_{>} (\vartheta\diamond\xi)$ are irregular. Then the singular term $u\diamond\xi := (\psi+\phi) \diamond \xi$ can be decomposed as
\begin{align*}
     & (\psi+\phi) \diamond \xi \\
    = & (\psi+\phi)\succ \xi + (\psi+\phi)\prec \xi  + (\psi+\phi) \circ \xi\\
    = & (\psi+\phi) \succ \mathscr{U}_{\leq} \xi + (\psi+\phi) \succ \mathscr{U}_{>} \xi + (\psi+\phi) \prec  \mathscr{U}_{\leq} \xi + (\psi+\phi) \prec \mathscr{U}_{>} \xi+ (\psi+\phi) \circ \xi.
\end{align*}
In order to define the resonant term $(\psi+\phi) \circ \xi$, we also need the modified paraproduct ansatz
\begin{equation*}
    \tau^{\gamma}\phi^{\sharp} = \tau^{\gamma}\phi-[\tau^{\gamma}(\psi+ \phi)\pprec \vartheta],
\end{equation*}
where $\phi^{\sharp}(t) \in \mathscr{C}^{2\alpha}$, and the modified paraproduct $\pprec$ is defined as (\ref{modifiedpara}). Then the resonant term can be defined as
\begin{align*}
    (\psi+\phi) \circ \xi  = & \psi\circ\xi +((\psi+\phi)\pprec\vartheta)\circ \xi + \phi^{\sharp} \circ\xi  \\
              = & \psi \circ \xi+ ((\psi+\phi)\pprec\vartheta - (\psi+\phi)\prec \vartheta)\circ\xi + C(\psi+\phi,\vartheta,\xi) +\phi^{\sharp}\circ\xi \\
              & +(\psi+\phi) \succ (\vartheta\diamond\xi) +(\psi+\phi)\circ (\vartheta\diamond\xi)+ (\psi+\phi)  \prec (\vartheta\diamond\xi).
\end{align*}
Now we define
\begin{align*}
    \Phi := &  (\psi+\phi) \prec \mathscr{U}_{>} \xi + (\psi+\phi) \succ \mathscr{U}_{>} \xi + (\psi+\phi) \succ \mathscr{U}_{>}(\vartheta\diamond\xi) - (\psi+\phi) \prec \mathscr{U}_{>}(\vartheta\diamond\xi), \\
    \Psi := & f(\psi+\phi) - f(\psi) + ((\psi+\phi)\pprec\vartheta - (\psi+\phi)\prec \vartheta)\circ\xi + \phi^{\sharp}\circ\xi \\
            & + C(\psi+\phi,\vartheta,\xi)+ (\vartheta\diamond\xi) \circ (\psi+\phi)  \\
            & + (\psi+\phi) \prec \mathscr{U}_{\leq} \xi + (\psi+\phi) \succ \mathscr{U}_{\leq}\xi +(\psi+\phi) \succ \mathscr{U}_{\leq}(\vartheta\diamond\xi) - (\psi+\phi) \prec \mathscr{U}_{\leq}(\vartheta\diamond\xi) .
\end{align*}

\subsection{A priori estimates}
{\bf Step 1. Bound for $\phi$ in $C_{\tau^{1/(k-2)+\kappa/2}}\mathscr{C}^{\kappa} \cap C_{\tau^{1/(k-2)}}L^{\infty}$ }\\
First, we estimate $\Phi$ in $C_{\tau^{1/(k-2)}}\mathscr{C}^{-2+\kappa}$, and derive a bound for $\phi$ in $C_{\tau^{1/(k-2)+\kappa/2}}\mathscr{C}^{\kappa} \cap C_{\tau^{1/(k-2)}}L^{\infty}$ by Schauder estimates.
By Lemma \ref{Localization}, we employ the Localization operators $\mathscr{U}_{\leq}$ and $\mathscr{U}_{>}$ with the parameter $L$ such that
\begin{equation*}
    \|\mathscr{U}_{>} \xi\|_{\mathscr{C}^{ -2+\kappa }} \lesssim 2^{-(1-2\kappa)L}\|\xi\|_{\mathscr{C}^{-1-\kappa}},
\end{equation*}
Then by Bony's paraproduct estimate, we have
\begin{align}
    & \|(\psi+\phi) \prec \mathscr{U}_{>} \xi\|_{C_{\tau^{1/(k-2)}}\mathscr{C}^{-2+\kappa}} + \|(\psi+\phi) \succ \mathscr{U}_{>} \xi\|_{C_{\tau^{1/(k-2)}}\mathscr{C}^{-2+\kappa}}\nonumber \\
   \lesssim   &  \|\mathscr{U}_{>} \xi\|_{\mathscr{C}^{-2+\kappa }} \|\psi+\phi\|_{C_{\tau^{1/(k-2)}}L^{\infty}} \nonumber \\
     \lesssim & 2^{-(1-2\kappa)L}\|\xi\|_{\mathscr{C}^{-1-\kappa}} \|\psi+\phi\|_{C_{\tau^{1/(k-2)}}L^{\infty}}.
\end{align}
Similarly, we employ the Localization operators $\mathscr{U}_{\leq}$ and $\mathscr{U}_{>}$ with the parameter $K$ such that
\begin{equation*}
    \|\mathscr{U}_{>} (\vartheta\diamond\xi)\|_{\mathscr{C}^{ -2+\kappa }} \lesssim 2^{-(2-3\kappa)K}\|\vartheta\diamond\xi\|_{\mathscr{C}^{-2\kappa}},
\end{equation*}
Then
\begin{align}
    & \|(\psi+\phi) \prec \mathscr{U}_{>} (\vartheta\diamond\xi)\|_{C_{\tau^{1/(k-2)}}\mathscr{C}^{-2+\kappa}} + \|(\psi+\phi) \succ \mathscr{U}_{>} (\vartheta\diamond\xi)\|_{C_{\tau^{1/(k-2)}}\mathscr{C}^{-2+\kappa}}\nonumber \\
   \lesssim   &  \|\mathscr{U}_{>} (\vartheta\diamond\xi)\|_{\mathscr{C}^{-2\kappa }} \|\psi+\phi\|_{C_{\tau^{1/(k-2)}}L^{\infty}} \nonumber \\
     \lesssim & 2^{-(2-3\kappa)K}\|\vartheta\diamond\xi\|_{\mathscr{C}^{-2\kappa}} \|\psi+\phi\|_{C_{\tau^{1/(k-2)}}L^{\infty}}.
\end{align}
Note that the stochastic terms $\xi$ and $\vartheta\diamond\xi$ can be constructed such that
\begin{equation*}
    \|\xi\|_{\mathscr{C}^{-1-\kappa}} \lesssim 1, \quad \|\vartheta\diamond\xi\|_{\mathscr{C}^{-2\kappa}} \lesssim 1.
\end{equation*}
Now we choose $L,K>1$, such that
\begin{equation*}
   1+ \|\psi + \phi\|_{C_{\tau^{1/(k-2)}}L^{\infty}} = 2^{(1 -\kappa)L} = 2^{(2-3\kappa)K}.
\end{equation*}
Then we have
\begin{equation}\label{EPhi0}
    \|\Phi\|_{C_{\tau^{1/(k-2)}}\mathscr{C}^{-2+\kappa}} \lesssim (2^{-(1 -\kappa)L}+2^{-(2-3\kappa)K})\| \psi +  \phi \|_{C_{\tau^{1/(k-2)}}L^{\infty}} \lesssim 1.
\end{equation}
Since
\begin{equation*}
    \mathscr{L}(\tau^{1/(k-2)+\kappa/2}\phi) = (\partial_t )\tau^{1/(k-2)+\kappa/2} \phi + \tau^{1/(k-2)-1+\kappa/2}(\partial_t \tau)\phi + \tau^{1/(k-2)+\kappa/2} \Phi,
\end{equation*}
by the Schauder estimates, we have the bound for $\phi$,
\begin{align*}
    \|\tau^{\kappa/2}\phi\|_{C_{\tau^{1/(k-2)}} \mathscr{C}^{\kappa}} \lesssim & \|\Phi\|_{C_{\tau^{1/(k-2)}}\mathscr{C}^{-2+\kappa}} + \|\tau^{1/(k-2)-\kappa/2}(\tau^{1+\kappa/2} \phi + \tau^{\kappa/2}(\partial_t \tau)\phi)\|_{CL^{\infty}} \\
    \lesssim & \|\Phi\|_{C_{\tau^{1/(k-2)}}\mathscr{C}^{-2+\kappa}} + \|\phi \|_{C_{\tau^{1/(k-2)}}L^{\infty}},
\end{align*}
Now we estimate $\phi$ in $C_{\tau^{1/(k-2)}}L^{\infty}$. Since $\tau^{1/(k-2)+\kappa/2}< \tau^{1/(k-2)}$, we can not control $\|\phi \|_{C_{\tau^{1/(k-2)}}L^{\infty}}$ by $\|\phi\|_{C_{\tau^{1/(k-2)+\kappa/2} }\mathscr{C}^{\kappa}}$ directly. By Littlewood-Paley decomposition and Duhamel's formula, for some $t\in (0,1)$ and $i \in \mathbb{N}$ we have
\begin{align*}
      & \|\tau(t)^{1/(k-2)}\phi(t) \|_{L^{\infty}} \nonumber \\
    \lesssim & \tau(t)^{1/(k-2)}\|\Delta_{\leq i}\phi \|_{L^{\infty}} + \tau(t)^{1/(k-2)}\|\Delta_{>i}\phi(t) \|_{L^{\infty}} \nonumber \\
     \lesssim & \tau(t)^{1/(k-2)}\|P_t \Delta_{\leq i}\phi(0)\|_{L^{\infty}} + \tau(t)^{1/(k-2)}\int^t_0 \|P_{t-s}\Delta_{\leq i}\Phi(s)\|_{L^{\infty}}ds + \tau(t)^{1/(k-2)}\|\Delta_{>i}\phi(t) \|_{L^{\infty}} \nonumber \\
     \lesssim & \tau(t)^{1/(k-2)}2^{i}\|\phi(0)\|_{\mathscr{C}^{-1}} + \tau(t)^{1/(k-2)} 2^{(2-\kappa)i}\|\Phi\|_{C\mathscr{C}^{-2+\kappa}} + 2^{-\kappa i}\tau(t)^{1/(k-2)}\|\phi(t)\|_{\mathscr{C}^{\kappa}}.
\end{align*}
We fix $t\in(0,1)$ and choose $i \in \mathbb{N}$ be such that $2^{-\kappa i} = \lambda\tau(t)^{\kappa/2}$ for any $\lambda>0$ which is independent on time. Then we have
\begin{equation*}
    \|\tau(t)^{1/(k-2)}\phi(t) \|_{L^{\infty}} \lesssim \tau(t)^{1/2}\tau(t)^{1/(k-2)}\|\phi(0)\|_{\mathscr{C}^{-1}} + \tau(t)^{1/(k-2)-\kappa/2}\|\Phi\|_{C\mathscr{C}^{-2+\kappa}} + \lambda \tau(t)^{1/(k-2)}\|\phi(t)\|_{\mathscr{C}^{\kappa}}.
\end{equation*}
Taking supremum in time, we obtain
\begin{equation}
    \|\phi\|_{C_{\tau^{1/(k-2)}}L^{\infty}} \lesssim \|\phi(0)\|_{\mathscr{C}^{-1}} + \|\Phi\|_{C_{\tau^{1/(k-2)}}\mathscr{C}^{-2+\kappa}} + \lambda \|\phi\|_{C_{\tau^{1/(k-2)+\kappa/2}}\mathscr{C}^{\kappa}}.
\end{equation}
Choosing $\lambda$ is small enough, we can absorb $\lambda \|\phi\|_{C_{\tau^{1/(k-2)+\kappa/2} }\mathscr{C}^{\kappa}}$ into the left hand side and obtain
\begin{equation}\label{Ephi1}
    \|\phi\|_{C_{\tau^{1/(k-2)+\kappa/2}}\mathscr{C}^{\kappa}} + \|\phi\|_{C_{\tau^{1/(k-2)}}L^{\infty}} \lesssim 1,
\end{equation}
where the right hand side is uniform in the initial condition $\|u_0\|_{\mathscr{C}^{-1}}$. We fix the parameters $L$ and $K$ in the remain part. We also have
\begin{equation}\label{Lep}
    2^{(1 -\kappa)L}= 2^{(2 - 3\kappa)K} = 1+\|
    (\psi + \phi)\|_{C_{\tau^{1/(k-2)}} L^{\infty}} \lesssim 1+\|\psi\|_{C_{\tau^{1/(k-2)}} L^{\infty}}.
\end{equation}
{\bf Step 2. Bound for $\phi$ in $C_{\tau^{1/(k-2)+\alpha/2}}\mathscr{C}^{\alpha} \cap C^{\alpha/2}_{\tau^{1/(k-2)+\alpha/2}}L^{\infty}$ }\\
First, we estimate $\Phi$ in $C_{\tau^{1/(k-2)}}\mathscr{C}^{-2+\alpha}$, and derive a bound for $\phi$ in $C_{\tau^{1/(k-2)+\alpha/2}}\mathscr{C}^{\alpha} \cap C^{\alpha/2}_{\tau^{1/(k-2)+\alpha/2}}L^{\infty}$ by Schauder estimates.
By Bony's paraproduct estimate, we have
\begin{align}
    & \|(\psi+\phi) \prec \mathscr{U}_{>} \xi + (\psi+\phi) \succ \mathscr{U}_{>} \xi\|_{C_{\tau^{1/(k-2)}}\mathscr{C}^{\alpha-2}} \nonumber \\
   \lesssim   &  \|\mathscr{U}_{>}\xi\|_{\mathscr{C}^{\alpha-2}} \|\psi+\phi\|_{C_{\tau^{1/(k-2)}}L^{\infty}} \nonumber \\
     \lesssim &  2^{-(1-\kappa-\alpha)L}\|\mathscr{U}_{>}\xi\|_{\mathscr{C}^{-1-\kappa}}\|\psi+\phi\|_{C_{\tau^{1/(k-2)}}L^{\infty}} \nonumber \\
     \lesssim & 1+\|\psi\|_{C_{\tau^{1/(k-2)}} L^{\infty}},
\end{align}
and
\begin{align}
    & \|(\psi+\phi) \succ \mathscr{U}_{>}(\vartheta\circ\xi) + (\psi+\phi) \prec \mathscr{U}_{>}(\vartheta\circ\xi)\|_{C_{\tau^{1/(k-2)}}\mathscr{C}^{\alpha-2}}  \nonumber \\
   \lesssim   &  \| \mathscr{U}_{>}(\vartheta\diamond\xi)\|_{\mathscr{C}^{\alpha-2 }} \|\psi+\phi\|_{C_{\tau^{1/(k-2)}}L^{\infty}} \nonumber \\
     \lesssim &  2^{-(2-2\kappa -\alpha)K}\|\mathscr{U}_{>}(\vartheta\circ\xi)\|_{\mathscr{C}^{-2\kappa}}\|\psi+\phi\|_{C_{\tau^{1/(k-2)}}L^{\infty}} \nonumber \\
      \lesssim & 1+\|\psi\|_{C_{\tau^{1/(k-2)}} L^{\infty}}.
\end{align}
Then we have
\begin{equation}\label{EPhi1}
    \|\Phi\|_{C_{\tau^{1/(k-2)}}\mathscr{C}^{\alpha-2}}  \lesssim 1 + \|\psi\|_{C_{\tau^{1/(k-2)}} L^{\infty}}.
\end{equation}

Now we estimate $\phi$ in $C_{\tau^{1/(k-2)+\alpha/2}}\mathscr{C}^{\alpha} \cap C^{\alpha/2}_{\tau^{1/(k-2)+\alpha/2}}L^{\infty}$ by Schauder estimates.
Since $\rho=\tau^{1/(k-2)}\eta$, we have
\begin{equation*}
    (\partial_t+\mathscr{L})(\tau^{1/(k-2)+\alpha/2} \phi) = (\partial_t \eta)\tau^{1/(k-2)+\alpha/2} \phi + \rho\tau^{-(2-\alpha)/2}(\partial_t \tau)\phi + \tau^{1/(k-2)+\alpha/2}  \Phi,
\end{equation*}
Then by the Schauder estimates, we obtain the $C_{\tau^{1/(k-2)+\alpha/2}}\mathscr{C}^{\alpha}$ bound for $\phi$,
\begin{align*}
    &\|\phi\|_{C_{\tau^{1/(k-2)+\alpha/2} }\mathscr{C}^{\alpha}} \\
    \lesssim & \|\Phi\|_{C_{\tau^{1/(k-2)+\alpha/2}}\mathscr{C}^{\alpha-2}} + \|\tau^{(2-\alpha)/2} (\tau^{1/(k-2)+\alpha/2} \phi + \tau^{1/(k-2)-(2-\alpha)/2} \phi)\|_{C_{\tau^{1/(k-2)}}L^{\infty}} \\
    \lesssim & \|\Phi\|_{C_{\rho}\mathscr{C}^{\alpha-2}} + \|\phi \|_{C_{\tau^{1/(k-2)}}L^{\infty}} \\
    \lesssim & 1 + \|\psi\|_{C_{\tau^{1/(k-2)}}L^{\infty}}.
\end{align*}
Using Lemma \ref{tauSch}, we have the time regularity
\begin{align*}
    & \|\phi\|_{C^{\alpha/2}_{\tau^{1/(k-2)+\alpha/2}}L^{\infty}} \\
    \lesssim & \|\phi\|_{C_{\tau^{1/(k-2)+\alpha/2}}\mathscr{C}^{\alpha}} + \|\Phi\|_{C_{\rho}\mathscr{C}^{\alpha-2}}  + \|\tau^{1/2} (\tau^{1/(k-2)+1/2} \phi + \tau^{1/(k-2)-1/2} \phi)\|_{C_{\eta}C^{\alpha-1}} \\
    \lesssim & 1 + \|\psi\|_{C_{\tau^{1/(k-2)}}L^{\infty}}.
\end{align*}
Thus we have a bound for $\phi$ in $C_{\tau^{1/(k-2)+\alpha/2}}\mathscr{C}^{\alpha} \cap C^{\alpha/2}_{\tau^{1/(k-2)+\alpha/2}}L^{\infty}$
\begin{equation}\label{phicalpha}
    \|\phi\|_{C_{\tau^{1/(k-2)+\alpha/2}}\mathscr{C}^{\alpha}} + \|\phi\|_{C^{\alpha/2}_{\tau^{1/(k-2)+\alpha/2}}L^{\infty}} \lesssim 1 + \|\psi\|_{C_{\tau^{1/(k-2)}}L^{\infty}}.
\end{equation}
{\bf Step 3. Bound for $\phi^{\sharp}$ in $C_{\rho}\mathscr{C}^{2\alpha}\cap C^{\alpha}_{\rho}L^{\infty}$}\\
Now we derive a bound for $\phi^{\sharp}$ in $C_{\rho}\mathscr{C}^{2\alpha}\cap C^{\alpha}_{\rho}L^{\infty}$. Recall that we denote $\rho =\tau^{1+1/(k-2)+(3\alpha-2)/2}$. Since $\phi^{\sharp}$ is given by
\begin{equation*}
    \phi^{\sharp} = \phi-\rho^{-1}\left([\rho(\psi+ \phi)]\pprec \vartheta \right),
\end{equation*}
the remainder $\phi^{\sharp}$ satisfies
\begin{align}\label{usharpg}
      (\partial_t+\mathscr{L})\phi^{\sharp} & =  :- (\partial_t+\mathscr{L})\left(\rho^{-1}[\rho(\psi+ \phi)\pprec \vartheta]\right) + \Phi  \nonumber \\
                   = & - [(\partial_t+\mathscr{L})(\rho^{-1}[\rho(\psi+ \phi)]\pprec\vartheta)- \rho^{-1}[\rho(\psi+ \phi)]\pprec (\partial_t+\mathscr{L})\vartheta] \nonumber \\
                    & + [ (\psi+\phi)\prec \xi-\rho^{-1}[\rho(\psi+ \phi)]\pprec (\partial_t+\mathscr{L})\vartheta ] - (\psi+\phi)\prec\xi + \Phi \nonumber \\
                   = & \left(\frac{k-1}{k-2}+\frac{3\alpha-2}{2}\right)(\tau^{-(\frac{k-1}{k-2}+\frac{3\alpha-2}{2})-1}(1-\tau)[\rho(\psi+ \phi)]\pprec\vartheta) \nonumber \\
                    & - \rho^{-1} \left( (\partial_t+\mathscr{L})[\rho(\psi+ \phi)]\pprec\vartheta)- [\rho(\psi+ \phi)]\pprec (\partial_t+\mathscr{L})\vartheta \right) \nonumber \\
                    & + \rho^{-1}\left([\rho(\psi+\phi)]\prec \xi-[\rho(\psi+ \phi)]\pprec \xi \right) \nonumber \\
                    & - (\psi+\phi)\prec\xi + \Phi.
\end{align}
Since $\vartheta = (-\Delta- \mu)^{-1}\xi$, the Schauder estimates yields that $\|\vartheta\|_{\mathscr{C}^{\alpha}} \lesssim \|\xi\|_{\mathscr{C}^{\alpha-2}} \lesssim 1$. Thus by Lemma \ref{pprec} we have
\begin{align}\label{Es1}
  & \|\tau^{(2-\alpha)/2}(\tau^{-(\frac{k-1}{k-2}+\frac{3\alpha-2}{2})-1}(1-\tau)[\rho(\psi+ \phi)]\pprec\vartheta)  \|_{C_{\rho}\mathscr{C}^{\alpha}} \nonumber \\
  \lesssim & \|\tau^{-1/2}[\rho(\psi+ \phi)]\pprec\vartheta \|_{C\mathscr{C}^{\alpha}} \nonumber \\
  \lesssim & \|(\psi+ \phi) \|_{C_{\tau^{1/(k-2)}}L^{\infty}}\|\vartheta\|_{\mathscr{C}^{\alpha}} \nonumber \\
  \lesssim & 1+ \|\psi\|_{C_{\tau^{1/(k-2)}}L^{\infty}}.
\end{align}
Lemma \ref{pprec} implies that
\begin{align}\label{Es2}
     & \|\rho^{-1} \left( (\partial_t+\mathscr{L})[\rho(\psi+ \phi)]\pprec\vartheta - [\rho(\psi+ \phi)]\pprec (\partial_t+\mathscr{L})\vartheta \right)\|_{C_{\rho}\mathscr{C}^{2\alpha-2}} \nonumber \\
    \lesssim & \|(\psi+\phi)\|_{C_{\rho}\mathscr{C}^{\alpha} }+\|(\psi+\phi)\|_{C^{\alpha/2}_{\rho}L^{\infty}},
\end{align}
and
\begin{align}\label{Es3}
     & \|\rho^{-1}\left([\rho(\psi+\phi)]\prec \xi-[\rho(\psi+ \phi)]\pprec \xi \right)\|_{C_{\rho}\mathscr{C}^{2\alpha}}  \nonumber \\
    \lesssim & \|\left([\rho(\psi+\phi)]\prec \xi-[\rho(\psi+ \phi)]\pprec \xi \right)\|_{C\mathscr{C}^{2\alpha-2}}  \nonumber \\
    \lesssim & \|\psi+\phi \|_{C_{\rho}\mathscr{C}^{\alpha} }+\| \psi+\phi \|_{C^{\alpha/2}_{\rho}L^{\infty}}.
\end{align}
Then by paraproduct estimates, we have
\begin{align}\label{Es4}
       & \|-(\psi+\phi)\prec\xi+\Phi\|_{C_{\rho}\mathscr{C}^{2\alpha -2} } \nonumber \\
    \lesssim & \|(\psi+\phi) \succ \mathscr{U}_{>} \xi + (\psi+\phi) \succ \mathscr{U}_{>}(\vartheta\circ\xi)\|_{C_{\rho}\mathscr{C}^{2\alpha -2} } \nonumber \\
    & +\|(\psi+\phi) \prec \mathscr{U}_{\leq} \xi  - (\psi+\phi) \prec \mathscr{U}_{>}(\vartheta\circ\xi)\|_{C_{\rho}\mathscr{C}^{2\alpha -2} } \nonumber \\
    \lesssim & (\| \mathscr{U}_{>} \xi\|_{\mathscr{C}^{\alpha-2}} + \| \mathscr{U}_{>} (\vartheta\circ\xi)\|_{\mathscr{C}^{\alpha-2}})\|(\psi+\phi)\|_{C_{\rho}\mathscr{C}^{\alpha} }   \nonumber \\
    &  +(\| \mathscr{U}_{\leq} \xi\|_{\mathscr{C}^{2\alpha-2}}+ \| \mathscr{U}_{>} (\vartheta\circ\xi)\|_{\mathscr{C}^{2\alpha-2}})\|(\psi+ \phi)\|_{C_{\tau^{1/(k-2)}}L^{\infty}} \nonumber \\
   \lesssim  & (2^{-(1-\kappa-\alpha)L}+2^{-(2-2\kappa-\alpha)K})\|(\psi+\phi)\|_{C_{\rho}\mathscr{C}^{\alpha} }  \nonumber \\
   & + (2^{(2\alpha-1+\kappa )L}+ 2^{-(2-2\kappa-2\alpha)K})\|(\psi+ \phi)\|_{C_{\tau^{1/(k-2)}}L^{\infty}} \nonumber \\
   \lesssim & \|(\psi+\phi)\|_{C_{\rho}\mathscr{C}^{\alpha} } + 1 + \|\psi \|^{1+\alpha}_{C_{\tau^{1/(k-2)}}L^{\infty}}.
\end{align}
Combining with above estimates (\ref{Es1})-(\ref{Es4}), and using the Schauder estimates, we have
\begin{align}\label{phisharp}
       &  \|\phi^{\sharp}\|_{C_{\rho}\mathscr{C}^{2\alpha}} + \|\phi^{\sharp}\|_{C^{\alpha}_{\rho}L^{\infty}}  \nonumber \\
    \lesssim & \|\tau^{(2-\alpha)/2}(\tau^{-(\frac{k-1}{k-2}+\frac{3\alpha-2}{2})-1}(1-\tau)[\rho(\psi+ \phi)]\pprec\vartheta)  \|_{C_{\rho}\mathscr{C}^{\alpha}} \nonumber \\
      & + \|\rho^{-1} \left( (\partial_t+\mathscr{L})[\rho(\psi+ \phi)]\pprec\vartheta - [\rho(\psi+ \phi)]\pprec (\partial_t+\mathscr{L})\vartheta \right)\|_{C_{\rho}\mathscr{C}^{2\alpha-2}} \nonumber \\
      & + \|\rho^{-1}\left([\rho(\psi+\phi)]\prec \xi-[\rho(\psi+ \phi)]\pprec \xi \right)\|_{C_{\rho}\mathscr{C}^{2\alpha-2}}  \nonumber \\
      & + \|-(\psi+\phi)\prec\xi+\Phi\|_{C_{\rho}\mathscr{C}^{2\alpha -2} } \nonumber \\
    \lesssim & \|\psi+\phi \|_{C_{\rho}\mathscr{C}^{\alpha} } + \|(\psi+ \phi)\|_{C_{\tau\eta}L^{\infty}} +\| \psi+\phi \|_{C^{\alpha/2}_{\rho}L^{\infty}}  \nonumber \\
    \lesssim & 1+\|\psi\|^{1+\alpha}_{C_{\tau^{1/(k-2)}}L^{\infty}} + \|\psi \|_{C_{\rho}\mathscr{C}^{\alpha} } +\|\psi\|_{C^{\alpha/2}_{\rho}L^{\infty}}.
\end{align}
{\bf Step 4. Bound for $\psi$ in $C_{\rho}\mathscr{C}^{3\alpha} \cap C^1_{\rho}L^{\infty}$}

Now we derive a bound for $\psi$ in  $C_{\rho}\mathscr{C}^{3\alpha} \cap C^1_{\rho}L^{\infty}$. Recall that we denote $\rho =\tau^{1+1/(k-2)+(3\alpha-2)/2}$.
By paraproduct estimates and a priori estimates (\ref{phicalpha}), (\ref{phisharp}), we have
\begin{align}\label{E41}
   \|\phi^{\sharp} \circ \xi\|_{C_{\rho}\mathscr{C}^{3\alpha -2}} \lesssim & \|\xi\|_{\mathscr{C}^{-1-\kappa}}\|\phi^{\sharp}\|_{C_{\rho}\mathscr{C}^{2\alpha}} \nonumber \\
    \lesssim & 1+\|\psi\|_{C_{\rho}\mathscr{C}^{\alpha}}  +\|\psi\|_{C^{\alpha/2}_{\rho}L^{\infty}}+ \|\psi\|_{C_{\tau^{1/(k-2)}}L^{\infty}}.
\end{align}
\begin{equation}\label{E42}
    \|\psi\circ\xi\|_{C_{\rho}\mathscr{C}^{3\alpha -2}}  \lesssim \|\psi\|_{C_{\rho}\mathscr{C}^{2\alpha }},
\end{equation}
\begin{align}\label{E43}
     \|\mathscr{U}_{\leq}(\vartheta\diamond\xi) \prec (\psi+\phi)\|_{{\mathscr{C}_{\rho} \mathscr{C}^{3\alpha -2}}}
    \lesssim & \|\vartheta\diamond\xi\|_{\mathscr{C}^{2\alpha -2}}\|\psi+\phi\|_{C_{\rho}\mathscr{C}^{\alpha}} \nonumber \\
    \lesssim &  1+\|\psi\|_{{C_{\rho} \mathscr{C}^{\alpha}}}+\|\psi\|_{C_{\tau^{1/(k-2)}}L^{\infty}}
\end{align}
\begin{align}\label{E44}
     \|(\vartheta\diamond\xi) \circ (\psi+\phi)\|_{{C_{\rho}\mathscr{C}^{3\alpha -2}}}
    \lesssim & \|\vartheta\diamond\xi\|_{\mathscr{C}^{2\alpha -2}}\|\psi+\phi\|_{C_{\rho}\mathscr{C}^{\alpha}} \nonumber \\
    \lesssim &  1+\|\psi\|_{{C_{\rho}\mathscr{C}^{\alpha}}}+\|\psi\|_{C_{\tau^{1/(k-2)}}L^{\infty}},
\end{align}
The commutator estimate Lemma \ref{commutatorE} implies that
\begin{align}
    \|C(\psi + \phi,\vartheta,\xi)\|_{{C_{\rho} \mathscr{C}^{3\alpha-2}}} \lesssim & \|\psi + \phi\|_{{C_{\rho} \mathscr{C}^{\alpha}}} \|\xi\|_{\alpha -2}\|\vartheta\|_{\alpha} \nonumber \\
    \lesssim & 1+\|\psi\|_{{C_{\rho} \mathscr{C}^{\alpha}}}+\|\psi\|_{C_{\tau^{1/(k-2)}}L^{\infty}}.
\end{align}
According to Lemma \ref{Localization} and the choosing of $L$ and $K$, we have
\begin{align}
     & \|(\psi+\phi)\prec \mathscr{U}_{\leq}(\vartheta\circ\xi) \|_{C_{\rho}\mathscr{C}^{3\alpha-2} }+ \|(\psi+\phi)\prec \mathscr{U}_{\leq} \xi \|_{C_{\rho}\mathscr{C}^{3\alpha-2} } \nonumber \\
     \lesssim & \|\psi + \phi\|_{C_{\tau^{1/(k-2)}}L^{\infty}}(\|\mathscr{U}_{\leq} (\vartheta\circ\xi)\|_{\mathscr{C}^{3\alpha-2}}+ \|\mathscr{U}_{\leq} \xi\|_{\mathscr{C}^{3\alpha-2}}) \nonumber \\
     \lesssim & 2^{(3\alpha-1+\kappa)L}\|\xi\|_{\mathscr{C}^{-1-\kappa}}\|\psi + \phi\|_{C_{\tau^{1/(k-2)}}L^{\infty}} + 2^{(3\alpha-2+2\kappa)K}\|\vartheta\circ\xi\|_{\mathscr{C}^{-2\kappa}}\|\psi + \phi\|_{C_{\tau^{1/(k-2)}}L^{\infty}} \nonumber \\
     \lesssim & 1+\|\psi\|^{3\alpha/(1-\kappa)}_{C_{\tau^{1/(k-2)}}L^{\infty}},
\end{align}
and
\begin{align}
       & \|(\psi+\phi) \succ \mathscr{U}_{\leq}(\vartheta\circ\xi)\|_{C_{\rho}\mathscr{C}^{3\alpha-2 }}+ \|(\psi+\phi) \succ \mathscr{U}_{\leq}\xi\|_{C_{\rho}\mathscr{C}^{3\alpha-2 }} \nonumber \\
        \lesssim & (\|\mathscr{U}_{\leq}(\vartheta\circ\xi)\|_{\mathscr{C}^{2\alpha-2}} + \|\mathscr{U}_{\leq}\xi\|_{\mathscr{C}^{2\alpha-2}}) \|\psi+ \phi\|_{{C_{\rho} \mathscr{C}^{\alpha}}} \nonumber\\
        \lesssim & (2^{(2\alpha-1+\kappa)L}+1)(1+\|\psi\|_{C_{\tau^{1/(k-2)}}L^{\infty}} +\|\psi\|_{{C_{\rho} \mathscr{C}^{\alpha}}}) \nonumber \\
           \lesssim  & (1+\|\psi\|^{-1+2\alpha/(1-\kappa)}_{C_{\tau^{1/(k-2)}}L^{\infty}})(1+\|\psi\|_{C_{\tau^{1/(k-2)}}L^{\infty}} +\|\psi\|_{{C_{\rho} \mathscr{C}^{\alpha}}}).
\end{align}
By (\ref{Ephi1}) and the dissipative assumption (\ref{assumf}) of $f$, we have
\begin{align}\label{E48}
    \|f(\psi+\phi)-f(\psi)\|_{C_{\rho}\mathscr{C}^{3\alpha-2}} & \lesssim \|f'(\psi+\phi)\|_{C_{\rho^{k-2}}L^{\infty}}\|\psi\|_{C_{\rho^{1+(k-2)(3\alpha-2)/2}}\mathscr{C}^{3\alpha-2}} \nonumber \\
    & \lesssim (1+\|\psi\|^{k-2}_{C_{\tau^{1/(k-2)}}L^{\infty}})\|\psi\|_{C_{\rho^{1+(k-2)(3\alpha-2)/2}}\mathscr{C}^{3\alpha-2}}.
\end{align}

Combining with above estimates, and using the interpolation result in Lemma \ref{interpolation} and weighted Young inequality, for every $\lambda>0$ we have
\begin{align}\label{Psi}
    &       \| \Psi \|_{C_{\rho}\mathscr{C}^{3\alpha -2}} \nonumber \\
    \lesssim &  1 +\|\psi\|_{C_{\rho}\mathscr{C}^{\alpha} }+ \|\psi\|_{C^{\alpha/2}_{\rho}L^{\infty} }+ \|\psi\|_{C_{\rho}\mathscr{C}^{2\alpha} }  +\|\psi\|^{\alpha-1+2\alpha/(1-\kappa)}_{C_{\tau^{1/(k-2)}}L^{\infty}} \nonumber\\
    & + \|\psi\|^{-1+2\alpha/(1-\kappa)}_{C_{\tau^{1/(k-2)}}L^{\infty}}\|\psi\|_{{C_{\rho} \mathscr{C}^{\alpha}}} + \|\psi\|^{k-2}_{C_{\tau^{1/(k-2)}}L^{\infty}}\|\psi\|_{C_{\rho^{1+(k-2)(3\alpha-2)/2}}\mathscr{C}^{3\alpha-2}} \nonumber\\
    \lesssim & 1 + \|\psi\|^{2/3}_{C_{\tau^{1/(k-2)}}L^{\infty}}\|\psi\|^{1/3}_{C_{\rho}\mathscr{C}^{3\alpha}}+ \|\psi\|^{1/2}_{C_{\tau^{1/(k-2)}}L^{\infty}}\|\psi\|^{1/2}_{C^{\alpha}_{\rho}L^{\infty}} + \|\psi\|^{1/3}_{C_{\tau^{1/(k-2)}}L^{\infty}}\|\psi\|^{2/3}_{C_{\rho}\mathscr{C}^{3\alpha}} \nonumber\\
            & +\|\psi\|^{\alpha-1+2\alpha/(1-\kappa)}_{C_{\tau^{1/(k-2)}}L^{\infty}} + \|\psi\|^{2/3-1+2\alpha/(1-\kappa)}_{C_{\tau^{1/(k-2)}}L^{\infty}}\|\psi\|^{ 1/3}_{C_{\rho}\mathscr{C}^{3\alpha}}  + \|\psi\|^{k-2+2/(3\alpha)}_{C_{\tau^{1/(k-2)}}L^{\infty}} \|\psi\|^{(3\alpha-2)/(3\alpha)}_{C_{\rho}\mathscr{C}^{3\alpha}} \nonumber \\
     \lesssim & 1+ \lambda \|\psi\|_{C_{\eta}\mathscr{C}^{3\alpha}} + \lambda\|\psi\|_{C^{\alpha}_{\eta}L^{\infty} } +\|\psi\|^{\alpha-1+2\alpha/(1-\kappa)}_{C_{\tau^{1/(k-2)}}L^{\infty}} \nonumber \\
     & +  \|\psi\|^{-1/2+3\alpha/(1-\kappa)}_{C_{\tau^{1/(k-2)}}L^{\infty}} +  \|\psi\|^{3\alpha(k-2)/2+1}_{C_{\tau^{1/(k-2)}}L^{\infty}}
\end{align}
Then by Schauder estimate Lemma \ref{nlSch} and choosing $\lambda$ small enough, we obtain
\begin{align}\label{psi}
& \|\psi\|_{C_{\rho}\mathscr{C}^{3\alpha}} + \|\psi\|_{C^1_{\rho}L^{\infty}} \nonumber \\
\lesssim & 1+ \| \Psi \|_{C_{\rho}\mathscr{C}^{3\alpha -2}} +\|\psi\|^{k+1}_{C_{\tau^{1/(k-2)}}L^{\infty}} \nonumber \\
\lesssim & 1 +\|\psi\|^{\alpha-1+2\alpha/(1-\kappa)}_{C_{\tau^{1/(k-2)}}L^{\infty}}
     +  \|\psi\|^{-1/2+3\alpha/(1-\kappa)}_{C_{\tau^{1/(k-2)}}L^{\infty}} +  \|\psi\|^{3\alpha(k-2)/2+1}_{C_{\tau^{1/(k-2)}}L^{\infty}} + \|\psi\|^{k-1}_{C_{\tau^{1/(k-2)}}L^{\infty}}
\end{align}
{\bf Step 5. Bound for $\psi$ in $C_{\tau^{1/(k-2)}}L^{\infty}$} \\
We estimate $\Psi$ in $C_{\tau^{1+1/(k-2)}}L^{\infty}$. Similar with estimates (\ref{E41})-(\ref{E44}), we have
\begin{align}\label{E51}
      \|\phi^{\sharp} \circ \xi\|_{C_{\tau^{1+1/(k-2)}}L^{\infty}} \lesssim & \|\xi\|_{\mathscr{C}^{-1-\kappa}}\|\phi^{\sharp}\|_{C_{\rho}\mathscr{C}^{2\alpha}} \nonumber \\
    \lesssim & 1+\|\psi\|_{C_{\rho}\mathscr{C}^{\alpha}}  +\|\psi\|_{C^{\alpha/2}_{\rho}L^{\infty}}+ \|\psi\|_{C_{\tau^{1/(k-2)}}L^{\infty}}.
\end{align}
\begin{equation}\label{E52}
    \|\psi\circ\xi\|_{C_{\tau^{1+1/(k-2)}}L^{\infty}}  \lesssim \|\psi\|_{C_{\rho}\mathscr{C}^{2\alpha }},
\end{equation}
\begin{align}
     \|\mathscr{U}_{\leq}(\vartheta\diamond\xi) \prec (\psi+\phi)\|_{{\mathscr{C}_{\tau^{1+1/(k-2)}}L^{\infty}}}
    \lesssim & \|\vartheta\diamond\xi\|_{\mathscr{C}^{-2\kappa}}\|\psi+\phi\|_{C_{\rho}\mathscr{C}^{\alpha}} \nonumber \\
    \lesssim &  1+\|\psi\|_{{C_{\rho} \mathscr{C}^{\alpha}}}+\|\psi\|_{C_{\tau^{1/(k-2)}}L^{\infty}}
\end{align}
\begin{align}
     \|(\vartheta\diamond\xi) \circ (\psi+\phi)\|_{{C_{\tau^{1+1/(k-2)}}L^{\infty}}}
    \lesssim & \|\vartheta\diamond\xi\|_{\mathscr{C}^{-2\kappa}}\|\psi+\phi\|_{C_{\rho}\mathscr{C}^{\alpha}} \nonumber \\
    \lesssim &  1+\|\psi\|_{{C_{\rho}\mathscr{C}^{\alpha}}}+\|\psi\|_{C_{\tau^{1/(k-2)}}L^{\infty}},
\end{align}
The commutator estimate Lemma \ref{commutatorE} implies that
\begin{align}
    \|C(\psi + \phi,\vartheta,\xi)\|_{{C_{\tau^{1+1/(k-2)}}L^{\infty}}} \lesssim & \|\psi + \phi\|_{{C_{\rho} \mathscr{C}^{\alpha}}} \|\xi\|_{-1-\kappa}\|\vartheta\|_{1-\kappa} \nonumber \\
    \lesssim & 1+\|\psi\|_{{C_{\rho} \mathscr{C}^{\alpha}}}+\|\psi\|_{C_{\tau^{1/(k-2)}}L^{\infty}}.
\end{align}
According to Lemma \ref{Localization} and the choosing of $L$ and $K$, we have
\begin{align}
     & \|(\psi+\phi)\prec \mathscr{U}_{\leq}(\vartheta\circ\xi) \|_{C_{\tau^{1+1/(k-2)}}L^{\infty} }+ \|(\psi+\phi)\prec \mathscr{U}_{\leq} \xi \|_{C_{\tau^{1+1/(k-2)}}L^{\infty} } \nonumber \\
     \lesssim & \|\psi + \phi\|_{C_{\tau^{1/(k-2)}}L^{\infty}}(\|\mathscr{U}_{\leq} (\vartheta\circ\xi)\|_{\mathscr{C}^{3\alpha-2}}+ \|\mathscr{U}_{\leq} \xi\|_{\mathscr{C}^{3\alpha-2}}) \nonumber \\
     \lesssim & 2^{(3\alpha-1+\kappa)L}\|\xi\|_{\mathscr{C}^{-1-\kappa}}\|\psi + \phi\|_{C_{\tau^{1/(k-2)}}L^{\infty}} + 2^{(3\alpha-2+2\kappa)K}\|\vartheta\circ\xi\|_{\mathscr{C}^{-2\kappa}}\|\psi + \phi\|_{C_{\tau^{1/(k-2)}}L^{\infty}} \nonumber \\
     \lesssim & 1+\|\psi\|^{3\alpha/(1-\kappa)}_{C_{\tau^{1/(k-2)}}L^{\infty}},
\end{align}
and
\begin{align}
       & \|(\psi+\phi) \succ \mathscr{U}_{\leq}(\vartheta\circ\xi)\|_{C_{\tau^{1+1/(k-2)}}L^{\infty}}+ \|(\psi+\phi) \succ \mathscr{U}_{\leq}\xi\|_{C_{\tau^{1+1/(k-2)}}L^{\infty}} \nonumber \\
        \lesssim & (\|\mathscr{U}_{\leq}(\vartheta\circ\xi)\|_{\mathscr{C}^{2\alpha-2}} + \|\mathscr{U}_{\leq}\xi\|_{\mathscr{C}^{2\alpha-2}}) \|\psi+ \phi\|_{{C_{\rho} \mathscr{C}^{\alpha}}} \nonumber\\
        \lesssim & (2^{(2\alpha-1+\kappa)L}+1)(1+\|\psi\|_{C_{\tau^{1/(k-2)}}L^{\infty}} +\|\psi\|_{{C_{\rho} \mathscr{C}^{\alpha}}}) \nonumber \\
           \lesssim  & (1+\|\psi\|^{-1+2\alpha/(1-\kappa)}_{C_{\tau^{1/(k-2)}}L^{\infty}})(1+\|\psi\|_{C_{\tau^{1/(k-2)}}L^{\infty}} +\|\psi\|_{{C_{\rho} \mathscr{C}^{\alpha}}}).
\end{align}
By (\ref{Ephi1}) and the dissipative assumption (\ref{assumf}) of $f$, we have
\begin{align}\label{E58}
    \|f(\psi+\phi)-f(\psi)\|_{C_{\tau^{1+1/(k-2)}}L^{\infty}} & \lesssim \|f'(\psi+\phi)\|_{C_{\rho^{k-2}}L^{\infty}}\|\psi\|_{C_{\tau^{1/(k-2)}}L^{\infty}} \nonumber \\
    & \lesssim (1+\|\psi\|^{k-2}_{C_{\tau^{1/(k-2)}}L^{\infty}})\|\psi\|_{C_{\tau^{1/(k-2)}}L^{\infty}}.
\end{align}
Combining with above estimates and (\ref{psi}), we obtain
\begin{equation*}
        \| \Psi \|_{C_{\tau^{1+1/(k-2)}}L^{\infty}} \lesssim  1+ \lambda  \|\psi\|^{k-1}_{C_{\tau^{1/(k-2)}}L^{\infty}}  +\|\psi\|^{-1/3+4/3(1-\kappa)}_{C_{\tau^{1/(k-2)}}L^{\infty}}+  \|\psi\|^{-1/2+2/(1-\kappa)}_{C_{\tau^{1/(k-2)}}L^{\infty}} +  \|\psi\|^{2(k-2)/2+1}_{C_{\tau^{1/(k-2)}}L^{\infty}}
\end{equation*}
for every $\lambda \in [0,1)$. Then by parabolic coercive estimates from Lemma \ref{nlSch} and weighted Young inequality, we obtain
\begin{equation}
    \|\psi\|_{C_{\tau^{1/(k-2)}}L^{\infty}} \lesssim 1+\|\Psi\|^{1/(k-1)}_{C_{\tau^{1/(k-2)}}L^{\infty}}  \lesssim 1.
\end{equation}

\subsection{Existence}

In this subsection, we prove the following existence result by a smooth approximation and compactness.

Let $u_{\epsilon}$ be a solution to the approximation equation
\begin{equation}
    \partial_t u_{\epsilon} + \mathscr{L}u_{\epsilon} = f(u_{\epsilon})+ u_{\epsilon}\diamond\xi_{\epsilon}, \quad u_{\epsilon}(0)=u_{0,\epsilon}.
\end{equation}
where $\xi_{\epsilon}\in C^{\infty}(\mathbb{T}^2)$ is the mollification of the spatial white noise $\xi$, $u_{\epsilon}\diamond \xi_{\epsilon}$ is the approximation of $ u\diamond \xi$, and $u_{0,\epsilon}$ is a smooth approximation of the initial value $u_0$. By Lemma \ref{Exapp}, for every $\epsilon\in (0,1)$ and $T>0$, there exists a unique classical solution $u_{\epsilon} \in C^{\infty}([0,T]\times\mathbb{T}^2)$ to the approximation equation.

\begin{thm}\label{Gexistence}
Let $ u_0\in \mathscr{C}^{-1}$, $\alpha \in [2/3,1)$. Then there exists a solution $(\phi, \psi, \phi^{\sharp})$ to system (\ref{decGPAM}) with
\begin{align*}
    \phi \in & [C_{\tau^{1/(k-2)+\alpha/2}}\mathscr{C}^{\alpha} \cap C^{\alpha/2}_{\tau^{1/(k-2)+\alpha/2}}L^{\infty}] \\
    \psi \in & [C_{\rho}\mathscr{C}^{3\alpha} \cap C^1_{\rho}L^{\infty} \cap C_{\tau^{1/(k-2)}}L^{\infty}] \\
    \phi^{\sharp} \in & [C_{\rho}\mathscr{C}^{2\alpha}\cap C^{\alpha}_{\rho}L^{\infty}],
\end{align*}
such that $u=\phi+\psi$ is a paracontrolled solution to the nonlinear parabolic Anderson model equation.
\end{thm}

\begin{proof}
Let $\xi_{\epsilon}$ be a smooth approximation of the spatial white noise $\xi$, and let $u_{0,\epsilon}$ be  a smooth approximation of the initial value $u_0$. Then by Lemma \ref{Exapp}, for every $\epsilon\in (0,1)$ and $T>0$, there exists a unique classical solution $u_{\epsilon} \in C^{\infty}([0,T]\times\mathbb{T}^2)$ to
\begin{equation}
    \partial_t u_{\epsilon}+\mathscr{L}u_{\epsilon} = f(u_{\epsilon})+ u_{\epsilon}\xi_{\epsilon}-C_{\epsilon}u_{\epsilon}, \quad u_{\epsilon}(0)=u_{0,\epsilon}.
\end{equation}
Where $c_{\epsilon}>0$ is the renormalization constant.
We decompose $u_{\epsilon}=\psi_{\epsilon}+\phi_{\epsilon}$ as same as above, such that the pair $(\psi_{\epsilon}, \phi_{\epsilon})$ satisfies the system,
\begin{equation}\label{appdecGPAM1}
  \left\{
   \begin{aligned}
   & \partial_{t}\phi_{\epsilon}+\mathscr{L}\phi_{\epsilon} = \Phi_{\epsilon}, \quad \phi_{\epsilon}(0) = \phi_{0,\epsilon}=u_{0,\epsilon}\\
   & \partial_{t}\psi_{\epsilon}+\mathscr{L}\psi_{\epsilon} = f(\psi_{\epsilon})+ \Psi_{\epsilon}, \quad \psi(0)=0,
   \end{aligned}
   \right.
\end{equation}
where the definitions of $\Phi_{\epsilon}$ and $\Psi_{\epsilon}$ are same as $\Phi$ and $\Psi$. Same as $\phi^{\sharp}$, we also define $\phi^{\sharp}_{\epsilon} = \phi- (\psi_{\epsilon} + \phi_{\epsilon})\pprec \vartheta_{\epsilon}$. From a priori estimates, for any $T>0$ the approximation $(\psi_{\epsilon}, \phi_{\epsilon}, \phi^{\sharp}_{\epsilon})$ have the following uniformly bounds  (uniformly in $\epsilon \in (0,1)$)
\begin{align}
    & \|\tau^{1/(k-2)+\alpha/2}\phi_{\epsilon}\|_{C_{T}\mathscr{C}^{\alpha}}+\|\tau^{1/(k-2)+\alpha/2}\phi_{\epsilon}\|_{C^{\alpha/2}_{T}L^{\infty}} \lesssim 1, \nonumber \\
    &  \|\rho\psi_{\epsilon}\|_{C_{T}\mathscr{C}^{3\alpha} } + \|\rho\psi_{\epsilon}\|_{C^1_{T}L^{\infty}} \lesssim 1 \nonumber \\
    &  \|\rho\phi_{\epsilon}^{\sharp}\|_{C_{T}\mathscr{C}^{2\alpha}} + \|\rho\phi_{\epsilon}^{\sharp}\|_{C^{\alpha}_{T}L^{\infty}} \lesssim 1 \nonumber
\end{align}
Due to the Besov embedding, Arzela-Ascoli theorem and Aubin-Lions argument, the space
\begin{equation*}
    [C_{T}\mathscr{C}^{\alpha} \cap C^{\alpha/2}_{T}L^{\infty}] \times [C_{T}\mathscr{C}^{3\alpha} \cap C^1_{T}L^{\infty} ]\times[C_{T}\mathscr{C}^{2\alpha} \cap C^{\alpha}_{T}L^{\infty}]
\end{equation*}
is compactly embedded into
\begin{equation*}
    [C_{T}\mathscr{C}^{\alpha-\delta} \cap C_{T}^{(\alpha-\delta)/2}\mathscr{C}^{-\gamma}] \times [C_{T}\mathscr{C}^{3\alpha-\delta} \cap C^{1-\delta}_{T}\mathscr{C}^{-\gamma} ]\times[C_{T}\mathscr{C}^{2\alpha-\delta} \cap C^{\alpha-\delta}_{T}\mathscr{C}^{-\gamma}]
\end{equation*}
provided $\delta \in (0,\alpha)$ and $\gamma \in (0,1)$ are chosen small. We refer Lemma 1 and Theorem 5 in [\refcite{S1987}] for more details. Thus there exists a convergent subsequence (still denoted $(\psi_{\epsilon}, \phi_{\epsilon}, \phi^{\sharp}_{\epsilon})$) which converge to some $(\psi, \phi, \phi^{\sharp})$ in above space.

Moreover, for any $T>0$, by linearity of the localizers $\mathscr{U}_{>}$, $\mathscr{U}_{\leq}$, and using same estimates in Section 3.2 we have
\begin{equation*}
    \rho\Phi_{\epsilon} \rightarrow \rho\Phi \quad \text{in} \quad C_T\mathscr{C}^{\alpha -2 - \delta}
\end{equation*}
and
\begin{equation*}
    \rho\Psi_{\epsilon} \rightarrow \rho\Psi \quad \text{in} \quad C_T\mathscr{C}^{3\alpha-2-\delta}.
\end{equation*}
Passing to the limit in (\ref{appdecGPAM1}). Thus limit $(\phi, \psi, \psi^{\sharp})$ solves the system (\ref{decGPAM}) in distributional sense.

Now we turn to show that
\begin{align*}
    \phi \in & [C_{\tau^{1/(k-2)+\alpha/2}}\mathscr{C}^{\alpha} \cap C^{\alpha/2}_{\tau^{1/(k-2)+\alpha/2}}L^{\infty}] \\
    \psi \in & [C_{\rho}\mathscr{C}^{3\alpha} \cap C^1_{\rho}L^{\infty} \cap C_{\tau^{1/(k-2)}}L^{\infty}] \\
    \phi^{\sharp} \in & [C_{\rho}\mathscr{C}^{2\alpha}\cap C^{\alpha}_{\rho}L^{\infty}],
\end{align*}
By a priori estimates for $(\phi_{\epsilon}, \psi_{\epsilon}, \phi_{\epsilon}^{\sharp})$, the Littlewood-Paley blocks $\Delta_i\phi_{\epsilon}$, $\Delta_i\psi_{\epsilon}$, $\Delta_i\phi_{\epsilon}^{\sharp}$ have uniform bounds
\begin{align*}
   & \|\tau(t)^{1/(k-2)+\alpha/2}\Delta_i\psi_{\epsilon}(t)\|_{L^{\infty}} \lesssim 1, \\
   & \|\rho(t)\Delta_i\phi_{\epsilon}(t)\|_{L^{\infty}}\lesssim 1, \\
   & \|\rho(t)\Delta_i\phi^{\sharp}_{\epsilon}(t)\|_{L^{\infty}} \lesssim 1
\end{align*}
uniform in $\epsilon$, $t$, and $i$. From weak $\ast$ lower semicontinuous of $L^{\infty}$ norm, we deduce that
\begin{align*}
    \|\tau(t)^{1/(k-2)+\alpha/2}\Delta_i\phi(t)\|_{L^{\infty}} \leq & \liminf_{\epsilon\rightarrow 0}\|\rho(t)\tau(t)^{\alpha/2}\Delta_i\phi_{\epsilon}(t)\|_{L^{\infty}} \\
    \leq & \liminf_{\epsilon\rightarrow 0}\|\phi_{\epsilon}\|_{C_{\tau^{1/(k-2)+\alpha/2}} \mathscr{C}^{\alpha}}2^{-i\alpha} \\
    \lesssim & 2^{-i\alpha} ,
\end{align*}
\begin{align*}
    \|\rho(t)\Delta_i\psi(t)\|_{L^{\infty}} \leq & \liminf_{\epsilon\rightarrow 0}\|\rho(t)\Delta_i\psi_{\epsilon}(t)\|_{L^{\infty}} \\
    \leq & \liminf_{\epsilon\rightarrow 0}\|\psi\|_{C_{\rho}\mathscr{C}^{3\alpha}}2^{-i3\alpha} \\
    \lesssim & 2^{-i3\alpha},
\end{align*}
\begin{align*}
    \|\rho(t)\Delta_i\psi(t)\|_{L^{\infty}} \leq & \liminf_{\epsilon\rightarrow 0}\|\rho(t)\Delta_i\psi_{\epsilon}(t)\|_{L^{\infty}} \\
    \leq &  \liminf_{\epsilon\rightarrow 0}\|\psi_{\epsilon}\|_{C_{\rho} L^{\infty}} \\
    \lesssim & 1,    
\end{align*}
\begin{align*}
    \|\rho(t)\Delta_i\phi^{\sharp}(t)\|_{L^{\infty}} \leq & \liminf_{\epsilon\rightarrow 0}\|\rho(t)\Delta_i\phi^{\sharp}_{\epsilon}(t)\|_{L^{\infty}} \\
    \leq & \liminf_{\epsilon\rightarrow 0}\|\phi^{\sharp}_{\epsilon}\|_{C_{\rho} \mathscr{C}^{2\alpha}}2^{-i2\alpha} \\
    \lesssim & 2^{-i2\alpha}.
\end{align*}
Above estimates imply that
\begin{equation*}
   (\phi, \psi, \phi^{\sharp}) \in  L^{\infty}_{\tau^{1/(k-2)+\alpha/2}}\mathscr{C}^{\alpha} \times [L^{\infty}_{\rho}L^{\infty} \cap L^{\infty}_{\rho}\mathscr{C}^{3\alpha}]\times L^{\infty}_{\rho}\mathscr{C}^{2\alpha}.
\end{equation*}
For the time regularity, we have
\begin{align*}
    \|\tau(t)^{1/(k-2)+\alpha/2}\phi(t)-\tau(s)^{1/(k-2)+\alpha/2}\phi(s)\|_{L^{\infty}} \lesssim & \liminf_{\epsilon\rightarrow 0}\|\tau(t)^{1/(k-2)+\alpha/2}\phi_{\epsilon}(t)-\tau(s)^{1/(k-2)+\alpha/2}\phi_{\epsilon}(s)\|_{L^{2}} \nonumber \\
    \lesssim & \|\phi_{\epsilon}\|_{C_{\tau^{1/(k-2)+\alpha/2}}^{\alpha/2}L^{2}}|t-s|^{\alpha/2} \nonumber \\
    \lesssim & |t-s|^{\alpha/2},
\end{align*}
\begin{align*}
    \|\rho(t)\psi(t)-\rho(s)\psi(s)\|_{L^{\infty}} \lesssim & \liminf_{\epsilon\rightarrow 0}\|\rho(t)\psi_{\epsilon}(t)-\rho(s)\phi_{\epsilon}(s)\|_{L^{2}} \nonumber \\
    \lesssim & \|\psi_{\epsilon}\|_{C^{1}_{\rho}L^{2}}|t-s| \nonumber \\
    \lesssim & |t-s|,
\end{align*}
and
\begin{align*}
    \|\rho(t)\phi^{\sharp}(t)-\rho(s)\phi^{\sharp}(s)\|_{L^{\infty}} \lesssim & \liminf_{\epsilon\rightarrow 0}\|\rho(t)\phi^{\sharp}_{\epsilon}(t)-\rho(s)\phi^{\sharp}_{\epsilon}(s)\|_{L^{2}} \nonumber \\
    \lesssim & \|\phi^{\sharp}_{\epsilon}\|_{C^{\alpha}_{\rho}L^{\infty}}|t-s|^{\alpha} \nonumber \\
    \lesssim & |t-s|^{\alpha}.
\end{align*}
Then we obtain time regularity. The proof is complete.
\end{proof}

\subsection{Uniqueness}

In this subsection, we consider the uniqueness of the nonlinear parabolic Anderson model equation (\ref{GPAM}) via the classical energy estimate.

\begin{thm}\label{Uniq}
The solution of (\ref{GPAM}) in the sense of Theorem \ref{Gexistence} is unique.
\end{thm}

\begin{proof}
Suppose $(\phi_1, \psi_1, \phi_1^{\sharp})$ and $(\phi_2, \psi_2, \phi_2^{\sharp})$ are two solutions of (\ref{GPAM}) which given in Theorem \ref{Gexistence}. Let $\zeta:= u_1 -u_2 = \psi_1 +\phi_1 -\psi_2 -\phi_2$, then $\zeta$ satisfies
\begin{equation}\label{zeta}
    \partial_{t}\zeta + \mathscr{L}\zeta-\zeta\diamond \xi= f(u_1) - f(u_2) , \quad \zeta(0)=0.
\end{equation}
Here, we use the simple paracontrolled $\zeta = \zeta\prec\vartheta + \zeta^{\sharp}$ to define $\zeta\diamond \xi$. Since $u =\phi+\psi = u \pprec\vartheta + \phi^{\sharp} + \psi $, the reminder $\zeta^{\sharp}$ is given by
\begin{align*}
   \zeta^{\sharp}:= & \zeta - \zeta\prec\vartheta \\
     = & (\psi_1 - \psi_2) + (\phi_1-\phi_2)-\zeta\prec\vartheta \\
     = & (\psi_1 - \psi_2) + ((\phi_1-\phi_2)\pprec \vartheta - (\phi_1-\phi_2)\prec \vartheta) - (\phi_1^{\sharp}-\phi_2^{\sharp}).
\end{align*}
The a priori estimates for $(\phi,\psi,\phi^{\sharp})$ yields that $\zeta^{\sharp}(t) \in \mathscr{C}^{2\alpha} \hookrightarrow H^{2\alpha}$. Thus $\zeta\diamond\xi$ is  given as follows
\begin{equation}\label{Wickp}
    \zeta\diamond\xi = \zeta \prec \xi + \zeta \succ \xi + \zeta^{\sharp}\circ\xi + C(\zeta,\vartheta, \xi)+ \zeta (\vartheta\diamond\xi).
\end{equation}

Now we multiply equation (\ref{zeta}) by $\zeta$, and take the $H^{\alpha-1}(\mathbb{T}^2)$ inner product to obtain
\begin{equation}\label{testeq}
    \frac{1}{2}\partial_t\|\zeta\|^2_{H^{\alpha-1}} + \|\nabla\zeta\|^2_{H^{\alpha-1}}+ \mu\|\zeta\|^2_{H^{\alpha-1}}  = \langle \zeta, f(u_1) - f(u_2)\rangle_{H^{\alpha-1}} + \langle \zeta, \zeta \diamond \xi\rangle_{H^{\alpha-1}}.
\end{equation}
We begin to estimate $\langle \zeta, \zeta \diamond \xi\rangle_{H^{\alpha-1}} $. By (\ref{Wickp}), this term can be decomposed as
\begin{align*}
   & \langle \zeta , \zeta \diamond \xi\rangle_{H^{\alpha-1}}  \\
  = &  \langle \zeta , \zeta \prec \xi\rangle_{H^{\alpha-1}} + \langle\zeta, \zeta \succ \xi\rangle_{H^{\alpha-1}} + \langle\zeta, \zeta^{\sharp}\circ\xi\rangle_{H^{\alpha-1}} + \langle \zeta , C(\zeta,\vartheta, \xi) \rangle_{H^{\alpha-1}} + \langle\zeta , \zeta (\vartheta\diamond\xi)\rangle_{H^{\alpha-1}}.
\end{align*}
By Lemma \ref{interpolationH} and weighted Young inequality, we have
\begin{align}\label{UE1}
    \langle \zeta,  \zeta \prec \xi  \rangle_{H^{\alpha-1}} + \langle \zeta, \zeta \succ \xi \rangle_{H^{\alpha-1}}\leq & \|\zeta\|_{H^{2\alpha-1+\kappa}}(\|\zeta \prec \xi \|_{H^{-1-\kappa}}+\|\zeta \succ \xi \|_{H^{-1-\kappa}}) \nonumber \\
    \lesssim & \|\zeta\|_{H^{2\alpha-1+\kappa}}\|\xi\|_{\mathscr{C}^{-1-\kappa}}\|\zeta\|_{H^{\kappa}} \nonumber \\
    \lesssim & \delta\|\nabla\zeta\|^2_{H^{\alpha-1}} + C_{\delta}\|\zeta\|^2_{H^{\alpha-1}},
\end{align}
By paraproduct estimates and Lemma \ref{Dm}, we have
\begin{align}\label{UE2}
   \langle\zeta , \zeta^{\sharp}\circ\xi\rangle_{H^{\alpha-1}} \leq & \|\zeta\|_{H^{2\alpha-1+\kappa}}\|\zeta^{\sharp} \circ \xi \|_{H^{-1-\kappa}} \nonumber \\
    \lesssim & \|\zeta\|_{H^{2\alpha-1+\kappa}}\|\xi\|_{\mathscr{C}^{-1-\kappa}}\|\zeta^{\sharp}\|_{H^{\kappa}} \nonumber \\
    \lesssim & \|\zeta\|_{H^{2\alpha-1+\kappa}}\|\xi\|_{\mathscr{C}^{-1-\kappa}}\|\zeta-\zeta\prec\vartheta\|_{H^{\kappa}} \nonumber \\
    \lesssim & \delta\|\nabla\zeta\|^2_{H^{\alpha-1}} + C_{\delta}\|\zeta\|^2_{H^{\alpha-1}}.
\end{align}
By paraproduct estimates, commutator estimates, and weight Young inequality, we have
\begin{align}\label{UE4}
    \langle \zeta, C(\zeta,\vartheta, \xi) \rangle_{H^{\alpha-1}} \leq & \|\zeta\|^2_{H^{\alpha-1}}+\|C(u,\vartheta,\xi)\|^2_{H^{\alpha-1}}\nonumber \\
    \lesssim & \|\zeta\|^2_{H^{\alpha-1}} + \|\zeta\|^2_{H^{1/2}}\|\vartheta\|_{\mathscr{C}^{1-\kappa}}\|\xi\|_{\mathscr{C}^{-1-\kappa}} \nonumber \\
    \lesssim & \delta\|\nabla\zeta\|^2_{H^{\alpha-1}} + C_{\delta}\|\zeta\|^2_{H^{\alpha-1}},
\end{align}
and
\begin{align}\label{UE5}
    \langle\zeta(t), \zeta(t) (\vartheta\diamond\xi)\rangle_{H^{\alpha-1}} \leq & \|\zeta(t)\|_{H^{\alpha-1}} \|\zeta (\vartheta\diamond\xi)\|_{H^{\alpha-1}} \nonumber\\
    \lesssim & \|\zeta(t)\|^2_{H^{\alpha-1}} + \|\zeta\|^2_{H^{1/2}}\|\vartheta\diamond\xi\|^2_{\mathscr{C}^{-2\kappa}} \nonumber \\
    \lesssim & \delta\|\nabla\zeta\|^2_{H^{\alpha-1}} + C_{\delta}\|\zeta(t)\|^2_{H^{\alpha-1}}.
\end{align}
From above estimates (\ref{UE1})-(\ref{UE5}), we have
\begin{equation}\label{UEE1}
    \langle \zeta, \zeta \diamond \xi\rangle_{H^{\alpha-1}} \lesssim \delta\|\nabla\zeta(t)\|^2_{H^{\alpha-1}} + C_{\delta}\|\zeta(t)\|^2_{H^{\alpha-1}}.
\end{equation}
Moreover, the assumption of $f$ implies that
\begin{align}\label{UEE6}
    \langle \zeta, (f(u_1)-f(u_2)) \rangle_{H^{\alpha-1}} \leq   l\| \zeta(t)\|^2_{H^{\alpha-1}}.
\end{align}
Plugging estimates (\ref{UEE1}) and (\ref{UEE6}) into (\ref{testeq}), and choosing $\delta$ small enough to absorb $\|\nabla\zeta(t)\|^2_{H^{\alpha-1}}$ in left hand side, we finally obtain
\begin{equation}\label{testeq8}
    \frac{1}{2}\partial_t \|\zeta(t)\|^2_{H^{\alpha-1}} \leq \delta\|\nabla\zeta(t)\|^2_{H^{\alpha-1}} + C_{\delta}\|\zeta(t)\|^2_{H^{\alpha-1}}.
\end{equation}
Since $\zeta(0) = \zeta^{\sharp}(0)=0$, by Gr\"{o}nwall's inequality, we deduce that $\zeta(t)=\zeta^{\sharp}(t)=0$ for every $t>0$.

Since $\phi_1 - \phi_2$ satisfies the linear equation
\begin{equation*}
    \mathscr{L}(\phi_1 - \phi_2) = \zeta \prec \mathscr{U}_{>} \xi + \zeta \succ \mathscr{U}_{>} \xi + \zeta \succ \mathscr{U}_{>}(\vartheta\circ\xi) - \zeta \prec \mathscr{U}_{>}(\vartheta\circ\xi), \quad (\phi_1- \phi_2)(0)=0,
\end{equation*}
if $\zeta(t)=0$ for every $t>0$, then $\phi_1 = \phi_2$, $\psi_1=\psi_2$. Furthermore, note that $\zeta^{\sharp}$ is given by
\begin{align*}
    \zeta^{\sharp}:& =(\phi_1-\phi_2) - \zeta\prec\vartheta \\
    & =(\phi_1^{\sharp}-\phi_2^{\sharp})-\rho^{-1}\left((\rho\zeta)\prec \vartheta- (\rho\zeta)\pprec \vartheta \right).
\end{align*}
If $\zeta = \zeta^{\sharp}=0$, then $\phi_1^{\sharp}=\phi_2^{\sharp}$. Thus the solution of (\ref{GPAM}) is unique. 
\end{proof}

\section{Conclusion}

We have established the global well-posedness result for the nonlinear parabolic Anderson model equation in paracontrolled distribution frame-work and parabolic Schauder and coercive estimates. Furthermore, we have also proved the uniqueness by using direct energy estimates. 

We point out that another possible method for the nonlinear parabolic Anderson model equation is using some properties of Anderson Hamiltonian $\mathscr{H}$ and employing $L^2$ energy estimates directly. In [\refcite{GUZ2020}], the authors using this method to study semilinear Schr\"{o}dinger and Wave equations with Anderson Hamiltonian $\mathscr{H}$. But if we use this methods, the regularity of the solution is lower than our results, and we need further regularity estimates for the equation.

There are still some possible extensions of our results. In fact, the noise term $u\diamond\xi$ can be replaced by more general case, such as $g(u)\diamond\xi$. We could extend the domain $\mathbb{T}^2$ to the whole space. To study the parabolic Anderson model equation on $\mathbb{R}^{+}\times\mathbb{R}^2$ we have to use some spatial weight. We could also consider the equation in higher dimension $(d= 3)$ and more singular noise, such as the spatial time white noise. The dynamical properties of the parabolic Anderson model equation are also interesting to investigate in further works.

\section*{Acknowledgments}
We are very grateful to our reviewer for valuable comments and editor's help.


\end{document}